\documentclass[11pt]{article}
\usepackage{fullpage}
\usepackage{amsmath}
\usepackage{amssymb}
\usepackage{amsthm}
\usepackage{fancyhdr}
\usepackage{algorithm}
\usepackage{algorithmic}
\usepackage{graphicx}
\usepackage{multirow}
\usepackage[usenames]{color}

\newtheorem{theorem}{Theorem}[section]
\newtheorem{lemma}[theorem]{Lemma}

\newtheorem{proposition}[theorem]{Proposition}
\newtheorem{definition}[theorem]{Definition}

\newtheorem{example}[theorem]{Example}
\def\wht{\widehat{\theta}}

\def\hx{\widehat{x}}

\def\C{\mathbb{C}}
\def\E{\mathbb{E}}
\def\P{\mathbb{P}}
\def\N{\mathcal{N}}
\def\labs{\left|}
\def\rabs{\right|}
\def\la{\left\langle}
\def\ra{\right\rangle}
\def\ln{\left\|}
\def\rn{\right\|}
\def\lb{\left(}
\def\rb{\right)}
\def\lsb{\left[}
\def\rsb{\right]}
\def\lcb{\left\{}
\def\rcb{\right\}}
\def\rank{\operatorname{rank}}
\def\tr{\operatorname{trace}}
\def\subto{\mbox{s.t.}}
\def\A{\mathcal{A}}
\def\R{\mathbb{R}}
\def\diag{\operatorname{diag}}
\def\dist{\operatorname{dist}}
\def\sign{\operatorname{sign}}
\DeclareMathOperator*{\argmin}{arg\,min}

\title{Solving systems of phaseless equations via Kaczmarz methods: A proof of concept study}
\author{Ke Wei\thanks{Department of Mathematics, Hong Kong University of Science and Technology, Hong Kong (makwei@ust.hk).}}
\begin{document}
\maketitle
\begin{abstract}
We  study the Kaczmarz methods for solving systems of phaseless equations, i.e., the generalized phase retrieval problem. 
The methods extend the Kaczmarz methods for solving systems of linear equations by integrating a phase selection heuristic in each iteration and overall have the 
same per iteration computational complexity.
Extensive empirical performance comparisons establish the computational advantages of the Kaczmarz methods over other state-of-the-art
phase retrieval algorithms both in terms of the number of measurements needed for successful recovery and in terms of  computation time. 
Preliminary convergence analysis is presented for the randomized Kaczmarz methods. 
\end{abstract}


%-------------------------------------
\section{Introduction}\label{sec:intr}
\subsection{The problem of phase retrieval}
In general, phase retrieval  is about  recovering a vector from the magnitude measurements, or equivalently solving a system of phaseless equations:
\begin{equation}\label{eq:problem_setup}
y_r=\left|\la a_r,x\ra\right|^2,~r=1,\cdots,m,
\end{equation}
where  $x\in\C^n$ and $a_r\in\C^n$. Let $A\in\C^{m\times n}$ be a matrix whose rows are $\{a_r^*\}_{1\leq r\leq m}$ and $y=(y_1,\cdots,y_m)^T$. 
The set of quadratic equations in \eqref{eq:problem_setup} can be formulated as $\sqrt{y}=|Ax|$, where $\sqrt{y}:=(\sqrt{y_1},\cdots,\sqrt{y_m})^T$ and
$|Ax|:=(\left|\la a_1,x\ra\right|,\cdots,\left|\la a_m,x\ra\right|)^T$. 
Despite its simple form, phase retrieval  arises in a wide range of practical context such as X-ray crystallography \cite{Ha93a}, diffraction imaging \cite{Buetal07} and microscopy \cite{Mietal08}, where the detector cannot measure the phase of the optical wave directly but only its magnitudes.  

Let $\widehat{x}$ be a solution to the 
phase retrieval problem. Apparently, $\widehat{x}e^{i\theta}$ is also a solution for any $\theta\in[0,2\pi)$. Therefore
the uniqueness of the solution to \eqref{eq:problem_setup} can only be defined up to a global phase factor. It has been shown  that $2n-1$ generic
magnitude measurements suffice to determine a unique solution to \eqref{eq:problem_setup}  if the measurement vectors $\{a_r\}_{1\leq r\leq m}$  and decision variables $x$ are real-valued \cite{BaCaEd06a}, while  $4n-4$
measurements are sufficient \cite{CoEdHeVi2013a} for the complex measurements and variables. %ROBUST RECOVERY
\subsection{Existing methods for phase retrieval}
In this section, we review some of the existing methods for phase retrieval.
The classical algorithms for phase retrieval are Error Reduction (ER, also known as Gerchberg-Saxton) and its variants which were pioneered 
by Gerchberg and Saxton \cite{GeSaPhase} and Fienup \cite{fienup1,fienup2}. Error Reduction 
is an alternating projection algorithm between the image and signal spaces, see Alg.~\ref{alg:gesa}. In each iteration, the current estimate is first projected in the image space so that
the magnitudes of its image match the measurements and then the new estimate is obtained by a least squares fitting. The first two steps of Alg.~\ref{alg:gesa} can also be interpreted
as a phase selection heuristic which approximates the image phase of the solution by that of the current estimate. 
%Moreover, a priori knowledge about the signals can be included in the fitting step, for example
%sparsity and nonnegativity.  %When the known structures of the signals are incorporated, Alg.~\ref{alg:gesa} and its variants 
%can be viewed as non-convex POCS (projections onto convex sets), Dyksatra's and Douglas-Rachford algorithms \cite{nonconvex_pocs}.
\begin{algorithm}[htp]
\caption{Error Reduction (ER)}\label{alg:gesa}
%\textbf{Iteration}: During iteration $l=0,1,\cdots$, \textbf{do}
\begin{algorithmic}
\STATE\textbf{Initilization}: $x_0$
\FOR{$l=0,1,\cdots$}
\STATE 1. $\theta_l=\angle A x_l$, where the angle is computed for each entry of $A x_l$
\STATE 2. $z_l=\sqrt{y}\odot e^{i\theta_l}$, $\odot$ is the elementwise multiplication
\STATE 3. $x_{l+1}=\argmin\limits_x\ln Ax-z_l\rn_2=A^\dagger z_l$ 
\ENDFOR
\end{algorithmic}
\end{algorithm}
As suggested by its name, ER satisfies the residual reduction property $\ln |Ax_{l+1}|-\sqrt{y}|\rn_2\leq\ln |Ax_l|-\sqrt{y}\rn_2$
which follows straightforward from the short calculation:
\begin{eqnarray}\label{eq:er}
\ln |Ax_{l+1}|-\sqrt{y}\rn_2 &=&\ln |Ax_{l+1}|-|\sqrt{y}\odot e^{i\theta_l}|\rn_2\nonumber\\
&\leq&\ln Ax_{l+1}-\sqrt{y}\odot e^{i\theta_l}\rn_2\nonumber\\
&\leq&\ln Ax_l-\sqrt{y}\odot e^{i\theta_l}\rn_2\nonumber\\
&=&\ln |Ax_l|\odot e^{i\theta_l}-\sqrt{y}\odot e^{i\theta_l}\rn_2\nonumber\\
&=&\ln |Ax_l|-\sqrt{y}\rn_2,
\end{eqnarray}
where the first inequality follows from the  triangle inequality and the second inequality follows from the definition of $x_{l+1}$.
When $A$ corresponds to Fourier transform, the last two steps in Alg.~\ref{alg:gesa} essentially adjust the phase of the measurements and then 
seek a signal which fits the magnitude measurements. Moreover, if the signal is known to belong to a convex set $C$, the new estimate can be
obtained by further projecting $A^\dagger z_l$ onto $C$ as
\begin{equation*}
x_{l+1}=P_C\lb A^\dagger z_l\rb.
\end{equation*}

Let $\A$ be a linear map from $n\times n$ Hermitian matrices to $m$ dimensional vectors defined as
\begin{equation*}
\A(X)_r = a_r^*Xa_r, ~r = 1,\cdots,m.
\end{equation*} 
Then the magnitude measurements in \eqref{eq:problem_setup} can be lifted up as the linear measurements with respect to 
the rank one matrices of the form $X=xx^*$:
\begin{equation*}
y_r = |a_r^*x|^2 = (a_r^*x)(x^*a_r) = a_r^*Xa_r=\A(X)_r.
\end{equation*}
Based on this observation,  finding the solution to \eqref{eq:problem_setup} is equivalent to a rank minimization problem
\begin{eqnarray}\label{eq:min_rank}
\min && \rank(X)\nonumber\\
\subto &&\A(X) = y\\
&& X\succeq 0\nonumber
\end{eqnarray}
by noting that the optimal solution to \eqref{eq:min_rank} coincides with $\hx\hx^*$ when  $\hx$  is a solution to \eqref{eq:problem_setup}. %, denoted by $\hx$.
However,  rank minimization  is a non-convex optimization problem. 
 Inspired by recent work on matrix completion, Cand\`es et al. \cite{phaselift1} propose an approach called PhaseLift by replacing the rank
functional in \eqref{eq:min_rank} with the matrix trace norm and instead solving a semidefinite programming
\begin{eqnarray}\label{eq:min_trace}
\min && \tr(X)\nonumber\\
\subto &&\A(X) = y\\
&& X\succeq 0.\nonumber
\end{eqnarray}
For certain random models,  %the further equivalence between  \eqref{eq:min_rank} and \eqref{eq:min_trace} can be established for the (nearly) optimal number of measurements \cite{phaselift2, phaselift3, phaselift4}. 
 it can be established that \eqref{eq:min_rank} and \eqref{eq:min_trace} share the same optimal solution if the number of measurements is (nearly) proportional to the size of the signal \cite{phaselift2, phaselift3, phaselift4}. 
In \cite{phasecut}, a different semidefinite programming called PhaseCut is proposed
by splitting the phase and magnitude variables.

Very recently, a line search algorithm called Wirtinger Flow (\cite{wirtinger}, see Alg.~\ref{alg:wirtinger}) has been developed by applying gradient descent iterations to the loss function
\begin{equation}
f(x) = \frac{1}{4m}\ln |Ax|^2-y\rn_2^2.
\end{equation} 
%When $a_r ~(r=1,\cdots,m)$ are real valued vectors, $f(x)$ is a degree-four polynomial which is generally non-convex, for example $(x^2-1)^2,~x\in\R$. 
%Starting from a carefully selected initial point $x_0$, Wirtinger Flow iteratively decreases $f(x)$ along some gradient type descent direction, see Alg.~\ref{alg:wirtinger}.
In each iteration of Wirtinger Flow, the current estimator $x_l$ is updated along the gradient descent direction $-\nabla f(x_l)$ with the stepsize $\mu_l$.
\begin{algorithm}[htp]
\caption{Wirtinger Flow}\label{alg:wirtinger}

%\textbf{Iteration}: During iteration $l=0,1,\cdots$, \textbf{do}
\begin{algorithmic}
\STATE \textbf{Initilization}: $x_0$, $\mu_l$
\FOR{$l=0,1,\cdots$}
\STATE 1. $\nabla f(x_l) =\frac{1}{m}A^*\lb \lb|Ax_l|^2-y\rb\odot\lb Ax_l\rb\rb$
\STATE 2. $x_{l+1} = x_l-\frac{\mu_l}{\ln x_0\rn_2^2}\nabla f(x_l)$
\ENDFOR
\end{algorithmic}
\end{algorithm}
Although $f(x)$ is not a convex function globally,\footnote{When $a_r$ ($r=1,\cdots,m$) are real-valued vectors, $f(x)$ is a degree-four polynomial which is generally non-convex, for example $(x^2-1)^2,~x\in\R$.} it has been proven in \cite{wirtinger} that for certain random measurement vectors, with high probability $f(x)$ is  a strongly convex 
function in a neighbourhood of $\hx$ when the number of measurements is nearly proportional to the length of the measured vector. Therefore  exponential convergence (or linear convergence) can be established for Wirtinger Flow if the initial point $x_0$ is selected to be in this small
neighbourhood of $\hx$. 

Moreover, there have been many algorithms designed especially for compressive phase retrieval problems in which the signals are known to be sparse, see \cite{gamp,gespar} and references therein. 

\subsection{Kaczmarz methods for linear equations}\label{sec:kacz_linear}
The simple Kaczmarz method \cite{kaczmarz}, also known as the  Algebraic Reconstruction Technique (ART) \cite{art}, is an iterative projection algorithm which was initially
designed for solving a  system of linear equations $Ax=y$. In the $l$th iteration, if the $r$th row of $A$ is selected, then the new estimate $x_{l+1}$ is obtained by projecting current estimate $x_l$ 
onto the hyperplane determined by the linear equation $\la a_{r},x\ra=y_{r}$ as
\begin{equation*}
x_{l+1} = x_l + \frac{y_{r}-\la a_{r},x_l\ra}{\ln a_r\rn_2^2}a_{r}.
\end{equation*}
The deterministic version of the simple Kaczmarz method usually sweeps through the rows of $A$ in a cyclic manner. The corresponding convergence results that appear in
the literature are all based on the quantities of $A$ that are hard to compute and involve the expressions  which do not have 
a clear geometric meaning \cite{kacrate1, kacrate2,kacrate3}. Strohmer and Vershynin \cite{StVe09b} provide 
the first exponential convergence analysis for a randomized version of the simple Kaczmarz method in terms of the scaled condition number of $A$ defined as 
\begin{equation*}
k(A) =\| A\|_F\| A^\dagger\|_2,
\end{equation*}
where $A^\dagger$ denotes the Moore-Penrose pseudo-inverse of $A$. In each iteration, the randomized Kaczmarz method randomly picks up a row of $A$
with the probability proportional to the square norms of the rows, that is
\begin{equation}
\P\lcb \mbox{the }r\mbox{th row is selected}\rcb = \frac{\ln a_r\rn_2^2}{\ln A\rn_F^2},~r=1,\cdots,m.
\end{equation}

In the block Kaczmarz method \cite{Elf80}, instead of selecting a single row of $A$ in each iteration, a subset of rows  are selected, denoted by $A_{\Gamma}$. % Then the 
%current estimate is projected onto the solution space of the small linear system of equations indexed by the selected rows. Let $A_\Gamma$ be the subset of rows that are selected in the $l$th iteration. 
Then the current estimate $x_l$ is updated by being projected onto the solution space of $A_{\Gamma}x=b_{\Gamma}$,
\begin{equation*}
x_{l+1} = x_l+A_{\Gamma}^\dagger\lb y_{\Gamma}-A_{\Gamma}x_l\rb.
\end{equation*}
The exponential convergence of the randomized block Kaczmarz method can be similarly established provided the partition of the measurement matrix is known, see  \cite{NeTr12a}.

When the linear system is inconsistent, exponential convergence to a neighbourhood of the desired solution is established for 
the randomized simple and block Kaczmarz methods in \cite{Need10a} and \cite{NeTr12a}. In order to apply the
Kaczmarz methods to solve the least squares problems, the extended Kaczmarz methods are designed to simultaneously decrease the system inconsistency
by  orthogonal column projections and update the estimate by the Kaczmarz methods \cite{Popa99, ZoFr2013}. Additional references
for Kaczmarz method include \cite{ElNe11a,NeWa13a,LvLe10}.
In the Kaczmarz methods, the residual does not decrease monotonically. In each iteration, while the residual of the selected rows reduces
to zero, the residual of the other rows will increase. 

\subsection{Main contributions and outline}
The main contributions of this manuscript are: 
\begin{enumerate}
\item
We introduce a family of Kaczmarz methods to find the solutions to Eq.~\eqref{eq:problem_setup}. The new methods are developed by integrating a phase selection heuristic into the classical Kaczmarz methods for linear equations. 
\item 
Extensive numerical experiments demonstrate the feasibility of applying the Kaczmarz methods for linear equations to solve the generalized phase retrieval problems, and establish the computational advantages of the Kaczmarz methods over prior art. 
\end{enumerate}

The rest of this manuscript is organised as follows. In section~\ref{sec:kacz_pr}, we derive and discuss the properties of the Kaczmarz methods for phase retrieval.
In section~\ref{sec:numerics}, we present detailed numerical comparisons of the Kaczmarz methods with ER and Wirtinger Flow. The preliminary convergence analysis for the randomized
Kaczmarz methods is presented in section~\ref{sec:proof}. Section~\ref{sec:conclusion} concludes this manuscript with some future research directions.
%-------------------------------------
\section{Kaczmarz methods for phase retrieval}\label{sec:kacz_pr}
\subsection{The simple Kaczmarz method}\label{sec:simp_kacz}
%\subsection{The simple Kaczmarz method}
%\begin{algorithm}[htp]
%\caption{Simple Kaczmarz}\label{alg:simp_kacz}
%\textbf{Initilization}: $x_0$

%\textbf{Iteration}: During iteration $l=0,1,\cdots$, \textbf{do}
%\begin{algorithmic}[1]
%\STATE select a row of $A$, denoted by $a_r^*$, either in a deterministic manner or randomly
%\STATE $\theta_l=\angle\la a_r,x_l\ra$
%\STATE $x_{l+1} = x_l+\frac{\sqrt{y_r}e^{i\theta_l}-\la a_r,x_l\ra}{\ln a_r\rn_2^2}a_r$
%\end{algorithmic}
%\end{algorithm}
\begin{algorithm}[htp]
\caption{Simple Kaczmarz}\label{alg:simp_kacz}
\begin{algorithmic}
\STATE\textbf{Initilization}: $x_0$
\FOR{$l=0,1,\cdots$}
%\FOR{$k=1,\cdots,m$}
\STATE 1. select a row of $A$, denoted by $a_r^*$, either in a deterministic manner or randomly
\STATE 2. $\theta_l=\angle\la a_r,x_l\ra$
\STATE 3. $x_{l+1} = x_l+\frac{\sqrt{y_r}e^{i\theta_l}-\la a_r,x_l\ra}{\ln a_r\rn_2^2}a_r$
%\ENDFOR
\ENDFOR
\end{algorithmic}
\end{algorithm}

We first present the simple Kaczmarz method for the phase retrieval problem \eqref{eq:problem_setup}, see Alg.~\ref{alg:simp_kacz}. 
In each iteration of the simple Kaczmarz method, it firstly selects a row of the measurement matrix either in a deterministic (e.g., cyclic) manner or 
randomly, and then projects the current estimate $x_l$  onto the hyperplane 
\begin{equation*}
\lcb x:~\la a_r,x\ra = \sqrt{y_r}e^{i\theta_l}\rcb~\mbox{ with }~\theta_l =\angle\la a_r,x_l\ra\in[0,2\pi),
\end{equation*}
where the image phase of the solution is approximated by that of the current estimate.
The selection of $\theta_l$ can also be interpreted in another way. Suppose in the $l$th iteration, we pick
up an arbitrary $\theta\in[0,2\pi)$, and then project $x_l$ onto the hyperplane
\begin{equation*}
\lcb x:~\la a_r,x\ra = \sqrt{y_r}e^{i\theta}\rcb.
\end{equation*}
Similarly the projection is given by
\begin{equation*}
x_{l+1}^\theta = x_l+\frac{\sqrt{y_r}e^{i\theta}-\la a_r,x_l\ra}{\ln a_r\rn_2^2}a_r.
\end{equation*}
Among all the candidates of $x_{l+1}^\theta$, it can be easily verified that $x_{l+1}$ 
is the one that minimizes the distance between $x_{l+1}^\theta$ and $x_l$, or 
equivalently\footnote{Note when $\la a_r,x_l\ra=0$, any $\theta\in[0,2\pi)$ 
minimizes $\ln x_{l+1}^\theta-x_l\rn_2^2$, in which case we will set $\theta_l=0$ as is typical in the literature.}
\begin{equation*}
\theta_l = \argmin_\theta\ln x_{l+1}^\theta-x_l\rn_2.
\end{equation*}
Therefore $x_{l+1}$ is indeed the projection of $x_l$ onto the set of hyperplanes
\begin{equation*}
\lcb x:~\left|\la a_r,x\ra\right| = \sqrt{y_r}\rcb,
\end{equation*}
and Alg.~\ref{alg:simp_kacz} can be viewed as a non-convex Kaczmarz method. 
%--------------------------------------------
\subsection{The block Kaczmarz method}\label{sec:block_kacz}
\begin{algorithm}[htp]
\caption{Block Kaczmarz}\label{alg:block_kacz}

%\textbf{Iteration}: During iteration $l=0,1,\cdots$, \textbf{do}
\begin{algorithmic}
\STATE\textbf{Initilization}: $x_0$, partition $T=\lcb \Gamma_1,\cdots,\Gamma_{N_b}\rcb$ of the row indices $\{1,\cdots,m\}$
\FOR{$l=0,1,\cdots,$}
%\FOR{$k=1,\cdots,N_b$}
\STATE 1. select a block $\Gamma_r$ from $T$ either in a deterministic manner or randomly
\STATE 2. $\theta_l=\angle A_{\Gamma_r} x_l$, where the angle is computed for each entry of $A_{\Gamma_r} x_l$
\STATE 3. $x_{l+1} = x_l+A_{\Gamma_r}^\dagger\lb \sqrt{y_{\Gamma_r}}\odot e^{i\theta_l}-A_{\Gamma_r}x_l\rb$
%\ENDFOR
\ENDFOR
\end{algorithmic}
\end{algorithm}
The block Kaczmarz method for phase retrieval (Alg.~\ref{alg:block_kacz}) begins
with a partition of the measurement matrix into a number of blocks. In each iteration, it firstly selects a block, denoted by $\Gamma_r$\footnote{In the block Kaczmarz method,  we assume that the row submatrix $A_{\Gamma_r}$ is fat (i.e., $|\Gamma_r|\leq n$) following the literature of the block Kaczmarz method for linear equations; though without this assumption ER  can be viewed as a special instance of the block Kaczmarz method with only
one block.}, either 
deterministically or randomly  and then projects the current estimate $x_l$ onto the intersections of the hyperplanes determined
by
\begin{equation*}
\lcb x:~A_{\Gamma_r} x=y_{\Gamma_r}\odot e^{i\theta_l}\rcb, ~\mbox{ with }~ \theta_l=\angle A_{\Gamma_r} x_l.
\end{equation*}
The pseudo-inverse in the third step of Alg.~\ref{alg:block_kacz} returns a solution of minimum $\ell_2$ norm to 
an underdetermined least squares problem.
%When each block of $A_\Gamma$ only contains a single row, Alg.~\ref{alg:block_kacz} reduces to Alg.~\ref{alg:simp_kacz}. If the number of rows 
%in $A_\Gamma$ is greater or equal than $n$, that is $|\Gamma|\geq n$, Step $3$ in Alg.~\ref{alg:block_kacz} can be simplified to 
%$x_{l+1}=A_{\Gamma}^\dagger\lb y_{\Gamma}\odot e^{i\theta_l}\rb$. In particular, if there is only one block in $T$ (i.e., $N_b=1$), Alg.~\ref{alg:block_kacz}
%is indeed the Gerchberg-Saxton algorithm (Alg.~\ref{alg:simp_kacz}).
%\begin{equation*}
%x_{l+1} = x_l + A^\dagger(y\odot e^{i\theta_l}-Ax_l) = A^\dagger(y\odot e^{i \theta_l})
%\end{equation*}
%\begin{proposition}
%Suppose $\rank(A_\Gamma)=\min\lb|\Gamma|,n\rb$, $\ln |A_\Gamma x_{l+1}|-y_\Gamma\rn_2\leq\ln |A_\Gamma x_l|-y_\Gamma\rn_2$.
%\end{proposition}
%If each block in Alg.~\ref{alg:block_kacz} only contains a single row of $A$, it reduces to Alg.~\ref{alg:simp_kacz}. Moreover, 
%when $A_{\Gamma_r}$  consists of orthogonal rows, applying one step of the block Kaczmarz method is equivalent to applying $|\Gamma_r|$
%steps of the simple Kaczmarz method.
The block Kaczmarz method for phase retrieval has the following property. 
\begin{proposition}\label{prop1}
Let $A_{\Gamma_i}$ and $A_{\Gamma_j}$ be two block submatrices of $A$ and $A_{\Gamma_i\cup\Gamma_j} = \begin{bmatrix}A_{\Gamma_i}\\A_{\Gamma_j}\end{bmatrix}$  be their concatenation. Assume $\rank(A_{\Gamma_i\cup\Gamma_j})=|\Gamma_i|+|\Gamma_j|\leq n$ and $A_{\Gamma_i}^*A_{\Gamma_j}=0$. 
Then applying two iterations of the block Kaczmarz method to the blocks $\Gamma_i$ and $\Gamma_j$ one after another is equivalent to applying one iteration of the block Kaczmarz 
method to the block $\Gamma_{i}\cup\Gamma_j$.
\end{proposition}
\begin{proof}
Let $x_0\in\C^n$ be an arbitrary point. Two iterations of the block Kaczmarz method applied to the blocks $\Gamma_i$ and $\Gamma_j$ gives
\begin{eqnarray*}
&&x_1 = x_0 + A_{\Gamma_i}^\dagger\lb\sqrt{y_{\Gamma_i}}\odot e^{i\angle A_{\Gamma_i} x_0}-A_{\Gamma_i}x_0\rb\\
&&x_2 = x_1 + A_{\Gamma_j}^\dagger\lb\sqrt{y_{\Gamma_j}}\odot e^{i\angle A_{\Gamma_j} x_1}-A_{\Gamma_j}x_1\rb.
\end{eqnarray*}
One iteration of the block Kaczmarz method to the block $\Gamma_{i}\cup\Gamma_j$ gives
\begin{equation*}
x_{12} = x_0 + A_{\Gamma_i\cup\Gamma_j}^\dagger\lb\sqrt{y_{\Gamma_i\cup\Gamma_j}}\odot e^{i\angle A_{\Gamma_i\cup\Gamma_j} x_0}-A_{\Gamma_i\cup\Gamma_j}x_0\rb.
\end{equation*}
From the condition $A_{\Gamma_i}A_{\Gamma_j}^*=0$, we have 
\begin{equation*}
A_{\Gamma_i}A_{\Gamma_j}^\dagger=A_{\Gamma_i}A_{\Gamma_j}^*(A_{\Gamma_j}A_{\Gamma_j})^{-1}=0\mbox{ and }A_{\Gamma_j}A_{\Gamma_i}^\dagger=A_{\Gamma_j}A_{\Gamma_i}^*(A_{\Gamma_i}A_{\Gamma_i})^{-1}=0,
\end{equation*}
which leads to
\begin{equation*}
A_{\Gamma_i\cup\Gamma_j}^\dagger = \begin{bmatrix} A_{\Gamma_i}^\dagger&A_{\Gamma_j}^\dagger\end{bmatrix}.
\end{equation*}
Therefore
\begin{eqnarray*}
x_{12} &=&x_0 + A_{\Gamma_i\cup\Gamma_j}^\dagger\lb\sqrt{y_{\Gamma_i\cup\Gamma_j}}\odot e^{i\angle A_{\Gamma_i\cup\Gamma_j} x_0}-A_{\Gamma_i\cup\Gamma_j}x_0\rb\\
&=&x_0+\begin{bmatrix} A_{\Gamma_i}^\dagger&A_{\Gamma_j}^\dagger\end{bmatrix}\lb
\begin{bmatrix}\sqrt{y_{\Gamma_i}}\odot e^{i\angle A_{\Gamma_i} x_0}\\
\sqrt{y_{\Gamma_j}}\odot e^{i\angle A_{\Gamma_j}x_0} \end{bmatrix}-\begin{bmatrix}A_{\Gamma_i}x_0\\A_{\Gamma_j}x_0\end{bmatrix}\rb\\
&=&x_0 + A_{\Gamma_i}^\dagger\lb\sqrt{y_{\Gamma_i}}\odot e^{i\angle A_{\Gamma_i} x_0}-A_{\Gamma_i}x_0\rb\\
&&+A_{\Gamma_j}^\dagger\lb\sqrt{y_{\Gamma_j}}\odot e^{i\angle A_{\Gamma_j} x_0}-A_{\Gamma_j}x_0\rb\\
&=&x_1+A_{\Gamma_j}^\dagger\lb\sqrt{y_{\Gamma_j}}\odot e^{i\angle A_{\Gamma_j} x_0}-A_{\Gamma_j}x_0\rb\\
&=&x_1+A_{\Gamma_j}^\dagger\lb\sqrt{y_{\Gamma_j}}\odot e^{i\angle A_{\Gamma_j} x_1}-A_{\Gamma_j}x_1\rb=x_2,
\end{eqnarray*}
where the second to last equality follows from the fact $A_{\Gamma_j}x_0=A_{\Gamma_j}x_1$ since $A_{\Gamma_j}A_{\Gamma_i}^\dagger=0$.
\end{proof}
Proposition~\eqref{prop1} implies that if the  submatrix $A_{\Gamma_r}$ consists of $|\Gamma_r|$ orthogonal rows, then one iteration of the block Kaczmarz method %applied to this block 
is 
equivalent to $|\Gamma_r|$ iterations of the simple Kaczmarz method. % applied to each row of $A_{\Gamma_r}$ sequentially.  %This type of property is easy to understand for the Kaczmarz methods on linear equations 
 In the Kaczmarz methods for linear equations, it is trivial that successive projections onto a set of orthogonal affine spaces returns the projection onto the intersection of the spaces. Proposition~\eqref{prop1}
confirms that this property still holds  when applying the Kaczmarz methods to solve systems of quadratic equations, without being influenced by the phase selection heuristic.

A natural question arises 
following the discussion in section~\ref{sec:simp_kacz}: If $|\Gamma_r|>1$, whether $x_{l+1}$ is the projection of $x_l$ onto the non-convex set
\begin{equation}\label{eq:proj_set}
\lcb x:~\labs A_{\Gamma_r}x\rabs=\sqrt{y_{\Gamma_r}}\rcb?
\end{equation}
Since for an arbitrary phase factor $e^{i\theta}$, the projection of $x_l$ onto the set
$
\lcb x:~ A_{\Gamma_r}x=\sqrt{y_{\Gamma_r}}e^{i\theta}\rcb
$
is given by $x_{l+1}^\theta = x_l+A_{\Gamma_r}^\dagger\lb \sqrt{y_{\Gamma_r}}\odot e^{i\theta}-A_{\Gamma_r}x_l\rb$, 
this question can be reformulated as whether 
\begin{equation}\label{eq:proj_dist}
\min_\theta\ln A_{\Gamma_r}^\dagger\lb \sqrt{y_{\Gamma_r}}\odot e^{i\theta}-A_{\Gamma_r}x_l\rb\rn_2
\end{equation}
attains its minimum at $\theta_l$. Unfortunately the answer is negative unless $A^\dagger_{\Gamma_r}$ has 
some special structures, such as $A^\dagger_{\Gamma_r}$ is a column orthonormal matrix in which case we have 
\begin{equation*}
\ln A_{\Gamma_r}^\dagger\lb \sqrt{y_{\Gamma_r}}\odot e^{i\theta}-A_{\Gamma_r}x_l\rb\rn_2 = \ln \sqrt{y_{\Gamma_r}}\odot e^{i\theta}-A_{\Gamma_r}x_l\rn_2.
\end{equation*}
A simple counterexample in $\R^2$ is given below to support this argument. 
\begin{example}\label{example}
Let 
\begin{equation*}
A^\dagger_{\Gamma_r} = \begin{bmatrix}
2 & 1\\
1 & 0
\end{bmatrix},~
\sqrt{y_{\Gamma_r}}=\begin{bmatrix}
2\\1
\end{bmatrix},~
A_{\Gamma_r} x_l = \begin{bmatrix}
1\\ 1
\end{bmatrix}.
\end{equation*}
Then $e^{i\theta_l}=\sign(A_{\Gamma_r}x_l)=\begin{bmatrix}1&1\end{bmatrix}^T$ and the objective function of
\eqref{eq:proj_dist} evaluated at $e^{i\theta_l}$ is $\sqrt{5}$, which is clearly larger than $1$ when $e^{i\theta}=\begin{bmatrix}1&-1\end{bmatrix}^T$.
\end{example}
%Computing the  solution to \eqref{eq:proj_dist} is not tractable in general as it is  essentially a non-convex optimization problem
%\begin{equation*}
%\min_u\ln A_{\Gamma_r}^\dagger\lb \sqrt{y_{\Gamma_r}}\odot u-A_{\Gamma_r}x_l\rb\rn_2~\subto~|u|=1.
%\end{equation*}
Let $u=e^{i\theta}$. Then \eqref{eq:proj_dist} can be further reformulated as a non-convex optimization problem
\begin{equation*}
\min_u\ln A_{\Gamma_r}^\dagger\lb \sqrt{y_{\Gamma_r}}\odot u-A_{\Gamma_r}x_l\rb\rn_2~\subto~|u|=1.
\end{equation*}
%which is a non-convex optimization problem. 
 So in general computing the  solution to \eqref{eq:proj_dist} is not tractable.
However, the following theorem bounds the deviation of $x_{l+1}$ from the optimal projection of $x_l$ onto the feasible set \eqref{eq:proj_set} in terms of the condition number 
of the submatrix.
\begin{theorem}[Projection Error Analysis]\label{prop2}
Let $\widehat{\theta}$ be the minimum of \eqref{eq:proj_dist} and $x_{l+1}^{\widehat{\theta}}$ be the corresponding projection of $x_l$. The difference 
between $x_{l+1}^{\widehat{\theta}}$ and $x_{l+1}$ can be bounded as 
\begin{equation}\label{eq:diff_to_optimal}
\ln x_{l+1}^{\widehat{\theta}}-x_{l+1}\rn_2\leq \kappa\lb A_{\Gamma_r}\rb \ln \diag\lb e^{i\widehat{\theta}}-e^{i\theta_l}\rb\rn_\infty\ln\hx\rn_2,
\end{equation}
where $\kappa\lb A_{\Gamma_r}\rb=\frac{\sigma_{\max}\lb A_{\Gamma_r}\rb}{\sigma_{\min}\lb A_{\Gamma_r}\rb}$ denotes the condition number of $A_{\Gamma_r}$, and $\hx$
is a solution to the phase retrieval problem \eqref{eq:problem_setup}.
\end{theorem}
\begin{proof} The result follows from the calculation:
\begin{eqnarray}
\ln x_{l+1}^{\widehat{\theta}}-x_{l+1}\rn_2 &=&\ln A_{\Gamma_r}^\dagger\lb \sqrt{y_{\Gamma_r}}\odot\lb e^{i\wht}-e^{i\theta_l}\rb\rb\rn_2\nonumber\\
&=&\ln A_{\Gamma_r}^\dagger\diag\lb e^{i\wht}-e^{i\theta_l}\rb \labs A_{\Gamma_r}\hx\rabs\rn_2\nonumber\\
&\leq&\ln A_{\Gamma_r}^\dagger\rn_2\cdot \ln \diag\lb e^{i\wht}-e^{i\theta_l}\rb\rn_\infty\cdot \ln A_{\Gamma_r}\rn_2\cdot\ln\hx\rn_2\nonumber\\
&=&\kappa\lb A_{\Gamma_r}\rb \ln \diag\lb e^{i\wht}-e^{i\theta_l}\rb\rn_\infty\ln\hx\rn_2,\nonumber
\end{eqnarray}
where the second equality follows from the fact $|A_{\Gamma_r}\hx|=\sqrt{y_{\Gamma_r}}$ and $\diag\lb e^{i\wht}-e^{i\theta_l}\rb$ denotes a diagonal matrix with the diagonal entries being the vector $e^{i\wht}~-~e^{i\theta_l}$.
%and the inequality follows from 
\end{proof}
%where $\kappa\lb A_{\Gamma_r}\rb=\frac{\sigma_{\max}\lb A_{\Gamma_r}\rb}{\sigma_{\min}\lb A_{\Gamma_r}\rb}$ denote the condition number of $A_{\Gamma_r}$.
%-------------------------------------- 
\subsection{Convergence results for the randomized Kaczmarz methods}
Let $\widehat{x}\in\C^n$ be a solution to the phase retrieval problem \eqref{eq:problem_setup}. For any $x\in\C^n$, the
distance of $x$ to $\widehat{x}$ is defined as
\begin{equation*}
\dist(x,\hx)=\min_{\theta\in[0,2\pi)}\ln x-\hx e^{i\theta}\rn_2.
\end{equation*}
As previously stated in section~\ref{sec:kacz_linear}, the randomized Kaczmarz methods for linear equations have been well-studied recently and
the corresponding convergence results only depends on the quantities of the matrix which are widely used in  numerical linear algebra. 
In this section, we will present preliminary convergence analysis of the randomized Kaczmarz methods for phase retrieval. For ease of exposition, we assume the 
measurement matrix $A$ is {\em standardized}, that is each row of $A$ has unit $\ell_2$ norm. In addition, we assume the rows or submatrices are selected uniformly at random in each iteration.
\begin{theorem}\label{thm:simple}
Let $A\in\C^{m\times n}$ be a standardized matrix with full column rank and $y=|A\hx|^2$ for some $\hx\in\C^n$. For an arbitrary initial estimate $x_0$, the iterates $x_l ~(l\geq 0)$ produced by the randomized simple
Kaczmarz method satisfy
\begin{equation}
\E\lsb\dist^2(x_l,\hx)\rsb
\leq\lb 1-\frac{\sigma^2_{\min}(A)}{m}\rb^l\dist^2(x_0,\hx)+\frac{4m}{\sigma^2_{\min}(A)}\ln y\rn_\infty.
\end{equation}
%where $R=\max_{1\leq r\leq m} \frac{2\sqrt{y_r}}{\ln a_r\rn_2}$.
\end{theorem}
Basically, theorem~\ref{thm:simple} states that for any given initial point, the exponential convergence of the  randomized simple Kacmarz method can be guaranteed until the 
iterate reaches a neighbourhood of $\hx$. However, the size of the  neighbourhood is quite pessimistic. %, especially  as compared with the empirical results when the test signals can be successfully reconstructed. 
 For example, if $A$ is a matrix whose rows are independent spherical random vectors, $\sigma_{\min}^2(A)$ will be proportional to $m/n$ \cite{nonasym_rand} and the size of the neighbourhood is proportional to the size of the measured vector.  
The proof of theorem~\ref{thm:simple} will be deferred to section~\ref{sec:proof}.

To state a similar result for the randomized block Kaczmarz method, we need the following definition introduced in \cite{NeTr12a}. 
\begin{definition}
An $(N_b,\alpha,\beta)$ row paving of a matrix $A$ is a partition $T=\{\Gamma_1,\cdots,\Gamma_{N_b}\}$ of the row indices that satisfies
\begin{equation*}
\alpha\leq\sigma^2_{\min}(A_{\Gamma_r})\leq \sigma^2_{\max}(A_{\Gamma_r})\leq \beta ~\mbox{ for each }\Gamma_r\in T.
\end{equation*}
\end{definition}
With this definition, we have the following property for the randomized block Kaczmarz method.
\begin{theorem}\label{thm:block}
Let $A\in\C^{m\times n}$ be a standardized matrix with full column rank and $y=|A\hx|^2$ for some $\hx\in\C^n$. For an arbitrary initial estimate $x_0$, the iterates $x_l ~(l\geq 0)$ produced by the randomized block
Kaczmarz method satisfy
\begin{equation}
\E\lsb\dist^2(x_l,\hx)\rsb
\leq\lb 1-\frac{\sigma^2_{\min}(A)}{\beta N_b}\rb^l\dist^2(x_0,\hx)+\frac{4\beta}{\alpha\sigma^2_{\min}(A)}\ln y\rn_1.
\end{equation}
%where $R=\max_{1\leq r\leq m} \frac{2\sqrt{y_r}}{\ln a_r\rn_2}$.
\end{theorem}
The proof of theorem~\ref{thm:block} follows the proof of theorem~\ref{thm:simple} and is omitted for brevity. In \cite{NeTr12a}, the authors review several 
approaches of constructing good pavings such that the condition numbers of all the submatrices are uniformly small. From theorem~\ref{prop2}, we can see that the paving of $A$ also
has an effect on the approximations of the optimal projections.% in each iteration. 
%-------------------------------------
\section{Numerical experiments}\label{sec:numerics}
In this section, we present the empirical observations of the Kaczmarz methods as compared 
with  ER (Alg.~\ref{alg:gesa}) and  Wirtinger Flow (Alg.~\ref{alg:wirtinger}).
ER and Wirtinger Flow are selected due to their applicability for signals without assuming any a priori knowledge,  flexibility in terms of the types of  measurements,
and  reported efficiency as greedy algorithms. 
We only present the empirical results of the deterministic 
Kaczmarz methods which go through the rows or partitions of the measurement matrix cyclically. The randomized 
Kaczmarz methods which select the rows or submatrices with the probability proportional to their sizes exhibit similar
performance in our test cases.  All the tested algorithms are implemented in Matlab with the code for Wirtinger Flow
downloaded from the author's website. The tests are conducted on a desktop computer with 4-core 3.2GHz processors and 8GB memory\footnote{The codes and data to reproduce all the figures and tables in the manuscript can be downloaded from {\tt https://github.com/makwei/phase-kacz}.}.
\subsection{Implementation details}
%DEFINE ``CYCLE'', CGLS, STOPPING CRITERIOR,Perssimistic, FOR EXAMPLE
\subsubsection{Termination conditions}
ER and Wirtinger Flow will be terminated after they reach a maximum number of iterations or the relative residual is small, that is
\begin{equation}\label{eq:rel_res}
\frac{\ln |Ax_l|^2-y\rn_2}{\ln y\rn_2}\leq\epsilon_1
\end{equation}
for some $\epsilon_1>0$.
For the Kaczmarz methods, we define a {\em cycle} as the number of iterations the algorithms  take to touch each row of the measurement matrix once. So a cycle consists of $m$ iterations for the simple Kaczmarz method and $N_b$ iterations for the block Kaczmarz method. The Kaczmarz methods will be terminated after they reach a maximum number of cycles. To verify the 
relative residual criterion \eqref{eq:rel_res} in each iteration is somewhat expensive for the Kaczmarz methods as it requires a matrix-vector product involving the full matrix $A$. The typical approach is to evaluate it after a fixed number of cycles.  However, in our implementations, we will check whether 
\begin{equation}\label{eq:max_atom_res}
\max\frac{\labs|a_r^*x_l|^2-y_r\rabs}{|y_r|}\leq\epsilon_2~\mbox{ or }~\max\frac{\ln |A_{\Gamma_r}x_l|^2-y_{\Gamma_r} \rn_2}{\ln y_{\Gamma_r}\rn_2}\leq\epsilon_2
\end{equation}
is satisfied or not after each cycle for some $\epsilon_2>0$, where the maximum is taken over a cycle of iterations.  Note that $x_l$ is the iterate before the projection, so $|a_r^*x_l|^2\neq y_r$ and 
$|A_{\Gamma_r}x_l|^2\neq y_{\Gamma_r}$. 
\subsubsection{Solving the least squares subproblems}
For the unstructured measurement matrix, one can apply either direct methods or iterative methods
to solve the overdetermined least squares subprobblems in ER and the underdetermined least squares subproblems in the block Kaczmarz methods, depending
on the sizes and conditions of the matrices or submatrices. In our tests for the Gaussian measurement matrix, we use Matlab {\tt pcg} with warm starting to solve the 
overdetermined least squares subproblems in ER and Matlab {\tt backslash} to solve the underdetermined least squares subproblems in the block Kaczmarz methods.
\subsection{Random experiments}\label{sec:rand_tests}
\subsubsection{1D set-up}
The measured vectors $\hx$ are drawn from the Gaussian distribution, that is $\hx~\sim~\N(0,I_n)$ 
when $\hx\in\R^n$ and $\hx\sim\N(0,I_n)+i\N(0,I_n)$ when $\hx\in\C^n$. The algorithms are tested 
with three different measurement models:
\begin{itemize}
\item the Gaussian model with entries of $A$ drawn i.i.d from either $\N(0,1)$ for real signals or $\N(0,1/2)+i\N(0,1/2)$ for complex signals.
\item the unitary model with $A$ being the concatenation of unitary matrices:
\begin{equation*}
A = \begin{bmatrix}
Q_1\\
\vdots\\
Q_{\lceil \frac{m}{n}\rceil}
\end{bmatrix}
\end{equation*}
where $Q_{\ell},~\ell=1,\cdots,\lceil \frac{m}{n}\rceil-1$ are unitary matrices and $Q_{\lceil \frac{m}{n}\rceil}$ is a row submatrix of a unitary matrix. 
We generate each $Q_{\ell}$ from the QR factorization of a random Gaussian matrix.% with entries drawn i.i.d from either $\N(0,1)$ for real signals or $\N(0,1)+i\N(0,1)$ for complex signals.
\item the coded diffraction model for complex signals with 
\begin{equation}
Ax = \begin{pmatrix}
F(\bar{d}_1\odot x)\\
\vdots\\
F(\bar{d}_L\odot x)
\end{pmatrix},
\end{equation}
where $F$ denotes the $n\times n$ DFT matrix and $d_\ell\in\C^n,~\ell=1,\cdots,L$ are a series of coded diffraction patterns (CDPs). 
In this paper, we will consider the {\em octonary} pattern suggested in \cite{phaselift4,wirtinger}, but note that other
 patterns are also possible, see \cite{wirtinger,phasecut}. In the octonary pattern, every entry 
of $d_\ell$ is a random variable of the form $b_1b_2$, where $b_1$ takes the value in $\{1,-1,i,-i\}$ with equal probability
and $b_2$ takes the value of $\sqrt{2}/2$ or $\sqrt{3}$ with probability $4/5$ and $1/5$ respectively.
\end{itemize} 
For the coded diffraction model,  the pseudo-inverse of $A$  in Alg.~\ref{alg:gesa} is given by
\begin{equation*}
A^\dagger \begin{pmatrix}z_1\\ \vdots\\ z_L\end{pmatrix} = \sum_{\ell=1}^L\lb F^*z_\ell\rb\odot \lb d_\ell/\sum_{\ell=1}^L|d_\ell|^2\rb
\end{equation*}
where $z_\ell\in\C^n,~\ell=1,\cdots,L$.
In the block Kaczmarz method, we will assume that all the rows of the submatrix $A_\Gamma$ correspond to the same coded diffraction 
pattern, that is for any $x\in\C^n$,
\begin{equation*}
A_\Gamma x = F_\Gamma(d_{\ell}\odot x) \mbox{ for some }1\leq\ell\leq L.
\end{equation*} 
Consequently for any $z\in\C^{|\Gamma|}$, the pseudo-inverse of $A_\Gamma$ is given by
\begin{equation*}
A_\Gamma^\dagger z = \frac{1}{d_\ell}\odot \lb F_\Gamma^*z\rb.
\end{equation*}
In particular, if $A$ is partitioned into $L$ equal blocks with each block corresponding to a coded diffraction pattern, the estimate update  
in the block Kaczmarz method (Step $3$ of Alg.~\ref{alg:block_kacz})
can be simplified to 
\begin{equation*}
x_{l+1} = \frac{1}{d_\ell}\odot\lb F^*x_l\rb,~\ell =1,\cdots,L.
\end{equation*}
The resulted block Kaczmarz method can be viewed as applying the ER update to each coded diffraction pattern sequentially, which is also  known as the {\em iterative transform algorithm} in the optics community \cite{wavefront, liao}, though
 motivated from a different perspective here. However, the residual 
reduction property stated in \eqref{eq:er} does not hold in general, see the remark about the Kaczmarz methods for linear equations at the end of section~\ref{sec:kacz_linear}.
%\end{itemize} 
\subsubsection{Successful recovery rate}\label{sec:rate}
In the $1$D simulations, we set $n=256$. For the Gaussian and unitary models,  the algorithms are tested for $20$ values of $m$  selected according to $m=\mbox{round}(\delta\cdot n)$ with $\delta$ taking 
 $20$ equally spaced values from $2$ to $6$. For the coded diffraction model, the number of coded diffraction patterns varies from $2$ to $12$. In this subsection, we will use the initial point suggested in \cite{wirtinger} for all the tested algorithms, i.e., we set 
\begin{equation}\label{eq:eig_init}
x_0 = \sqrt{\frac{\sum_{r=1}^ny_r}{m}}z,
\end{equation}
where $z$ is the  unit leading eigenvector of $\sum_{r=1}^ny_ra_ra_r^*$. 
It has been proved in \cite{wirtinger} that for the Gaussian and coded diffraction models $x_0$ can be arbitrarily close to $\hx$ 
if  $m$ is the order of $n\log n$. Starting from the same initial point,
ER and Wirtinger Flow are terminated after  $2500$ iterations are reached or the relative residual $\ln |Ax_l|^2-y\rn_2/\ln y\rn_2$ is below $10^{-8}$, while the Kaczmarz methods are terminated after they have run $500$ cycles or the maximum relative residual defined in \eqref{eq:max_atom_res} is below $10^{-7}$. For every pair of $(n,m)$, $50$ random tests are conducted. We consider a vector to be successfully recovered if the algorithm returns a vector $x_l$ which has a small relative error, that is when 
\begin{equation*}\mbox{rel.err}:=\frac{\dist(x_l,\hx)}{\ln\hx\rn_2}\leq 10^{-5}.\end{equation*}

We test the block Kaczmarz method with four different block sizes
\begin{equation*}|\Gamma_r|\in\{n/8, n/4,n/2,n\}=\{32,64,128,256\}.
\end{equation*} The rows of the measurement matrix are partitioned into
equal blocks (except the last block) with each block containing $|\Gamma_r|$ consecutive rows. 
%For the unitary model, only the simple Kaczmarz method is tested because of Proposition~\ref{prop1}. We also conduct tests for the $2$D coded diffraction model \footnote{For the other two models, $2D$ tests are essentially the same as $1$D tests after vectorizing the vectors and measurement matrices.} in which the vectors $\hx\in\C^{n_1\times n_2}$ are sampled from Gaussian distribution
% $\N(0,I_{n_1\times n_2})+i\N(0,I_{n_1\times n_2})$. The set-up for the $2D$ coded diffraction model is similar to the $1$D case, but with the $1$D DFT replaced by the $2$D DFT.
%For the Kaczmarz methods, only the block variant with each block corresponding to a coded diffraction is tested due to its efficiency by using $2D$ FFT. In the $2$D tests, $n_1=n_2=256$.
The empirical probabilities of successful recoveries for different tested models and algorithms are plotted in figure~\ref{fig:prob_of_succ}.
We can make several observations out of the plots.

\begin{figure}[htp]
\centering
\begin{tabular}{cc}
\vspace{0.1in}
\includegraphics[height=1.8in,width=2.2in,trim=0.7in 3in 0.7in 3in]{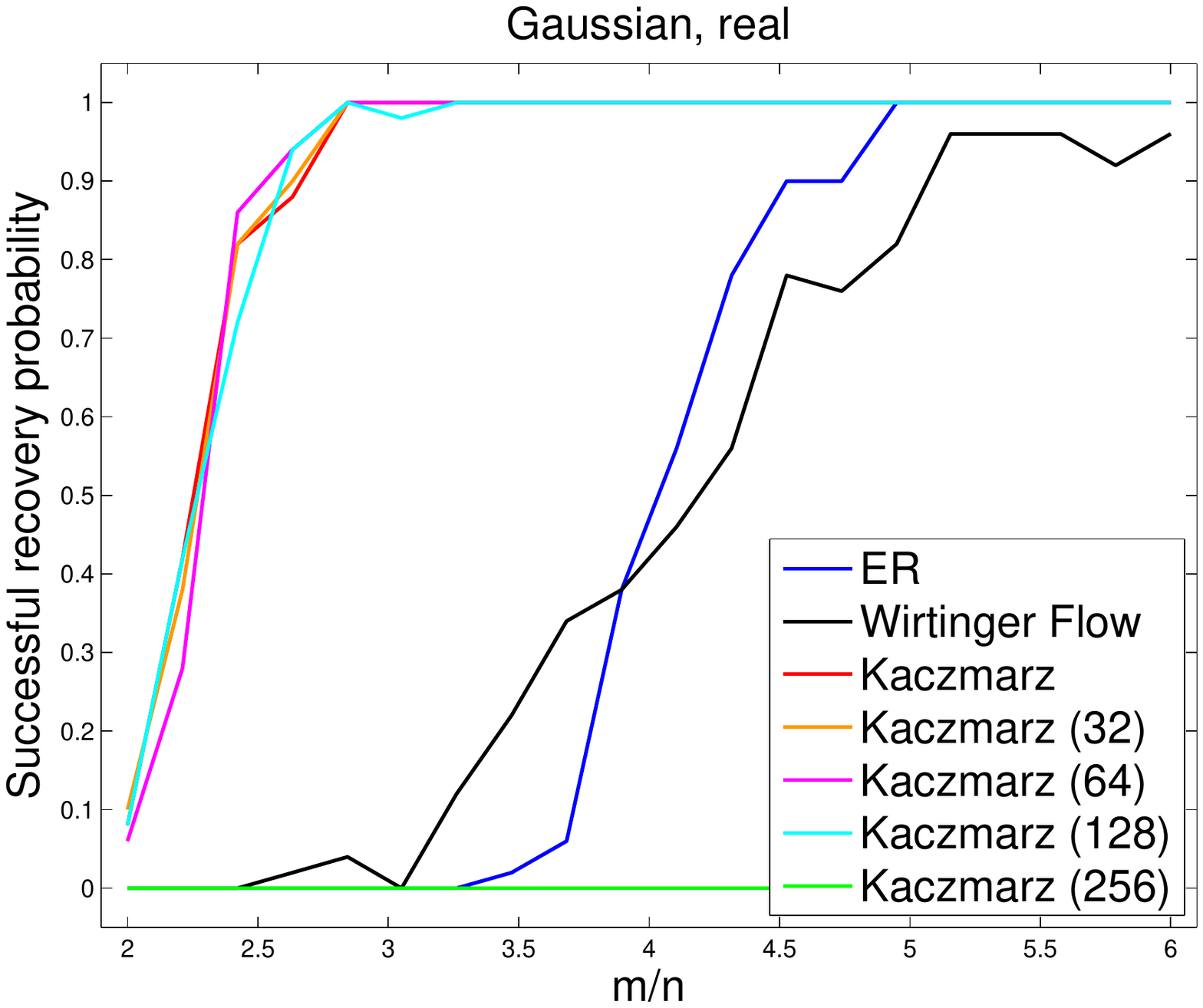}& 
\includegraphics[height=1.8in,width=2.2in,trim=0.7in 3in 0.7in 3in]{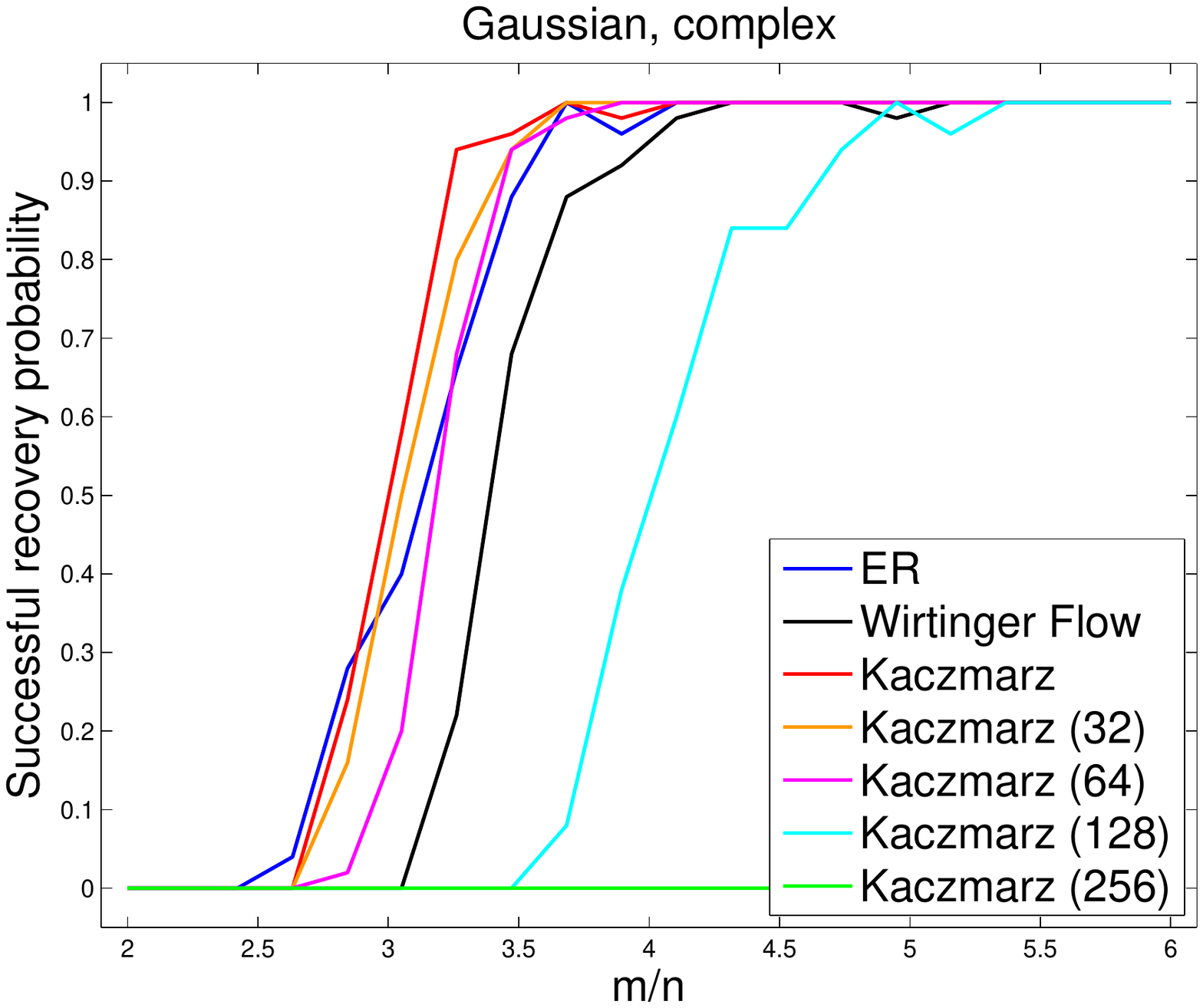}\\
\vspace{0.2in}
(a) & (b)\\
\vspace{0.1in}
\includegraphics[height=1.8in,width=2.2in,trim=0.7in 3in 0.7in 3in]{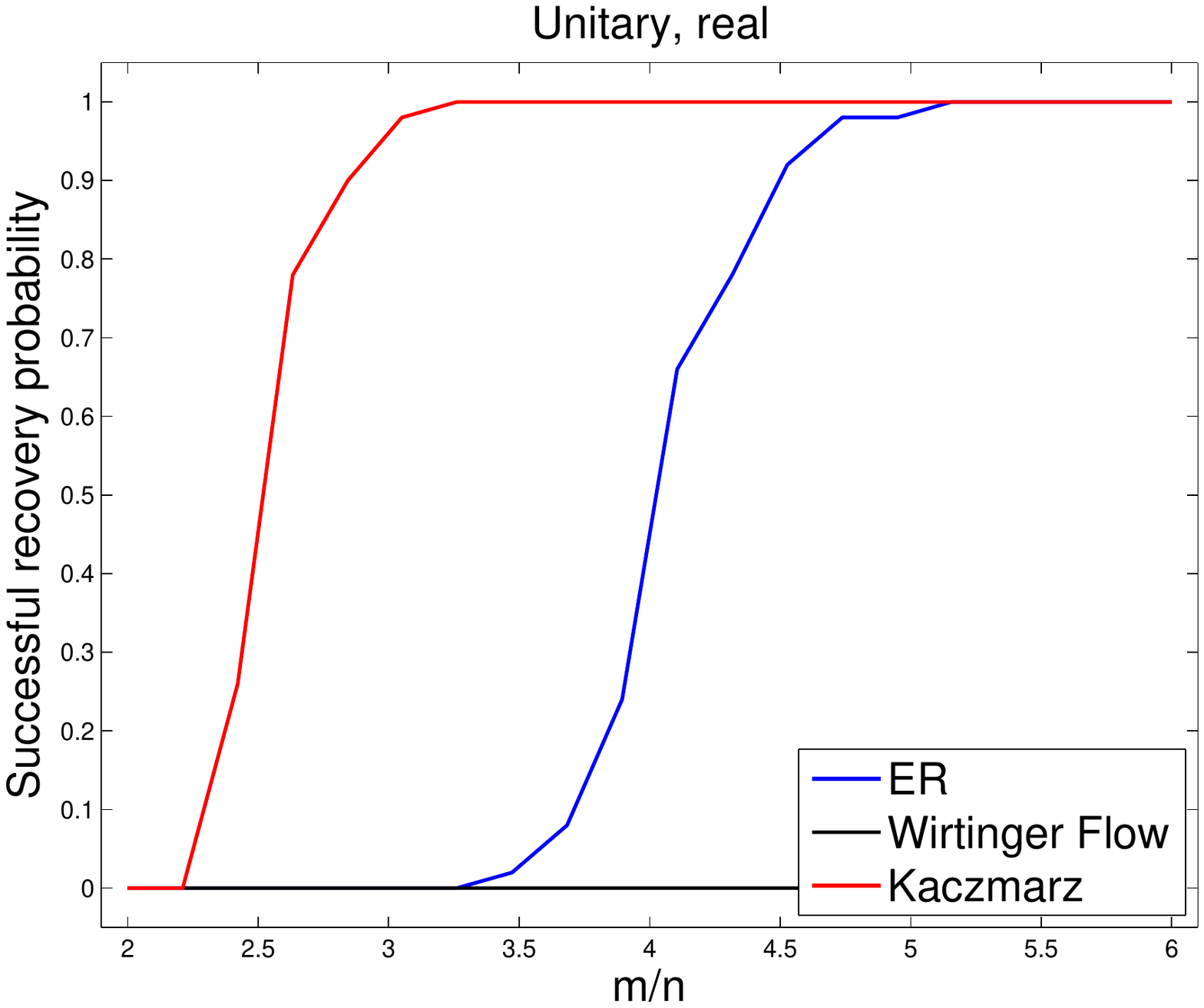}& 
\includegraphics[height=1.8in,width=2.2in,trim=0.7in 3in 0.7in 3in]{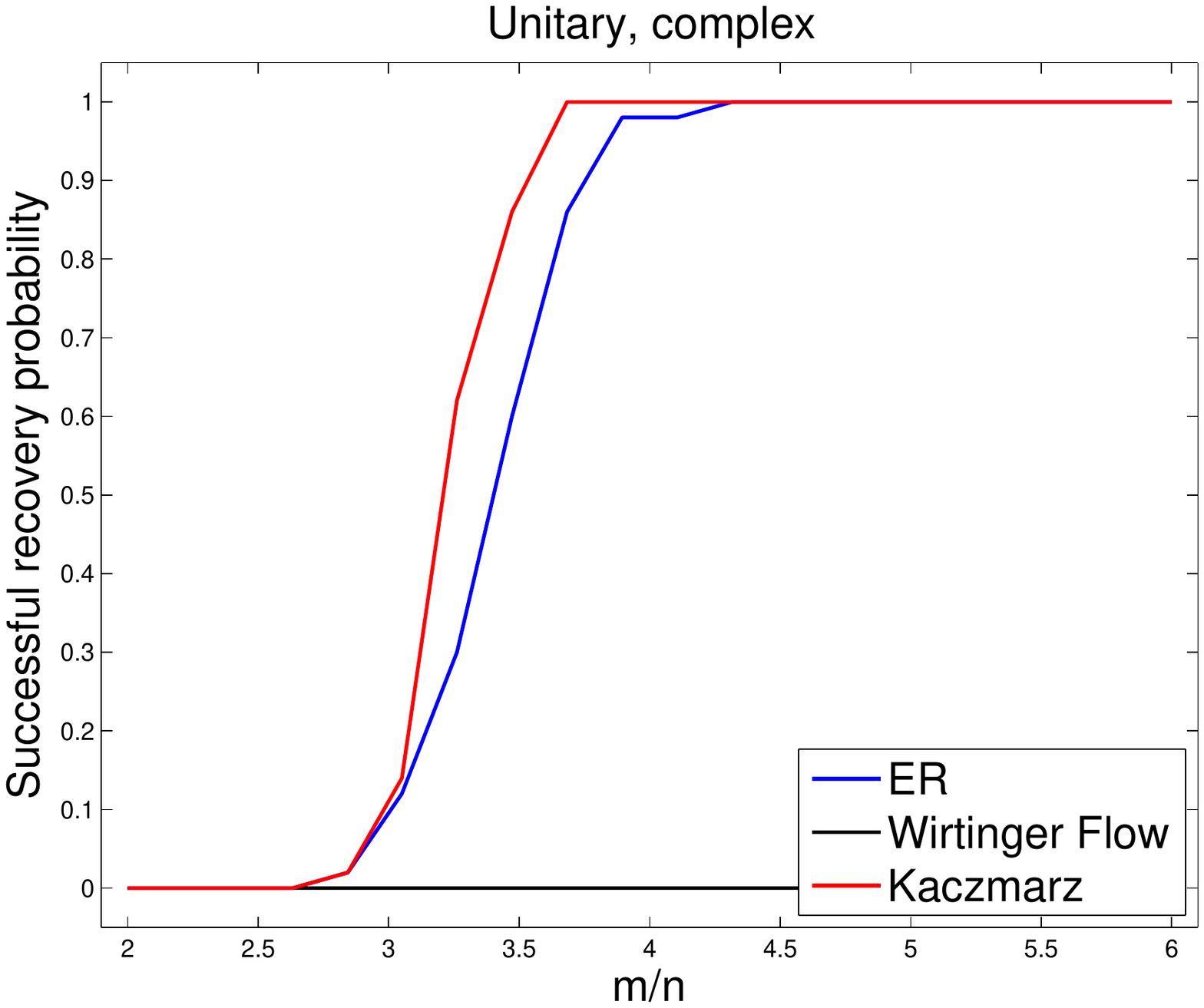}\\
\vspace{0.2in}
(c) & (d) \\
\vspace{0.1in}
\includegraphics[height=1.8in,width=2.2in,trim=0.7in 3in 0.7in 3in]{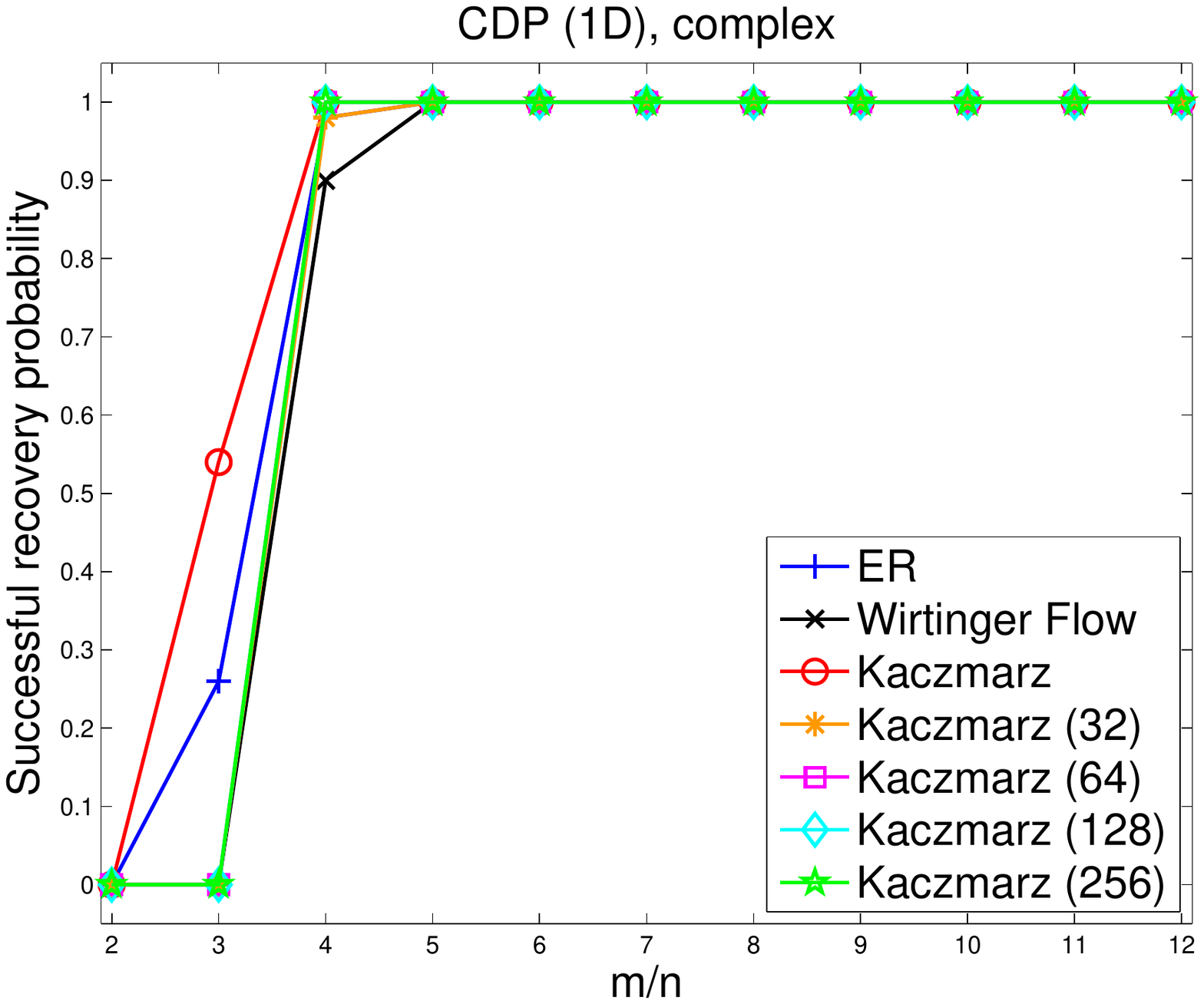}& 
\includegraphics[height=1.8in,width=2.2in,trim=0.7in 3in 0.7in 3in]{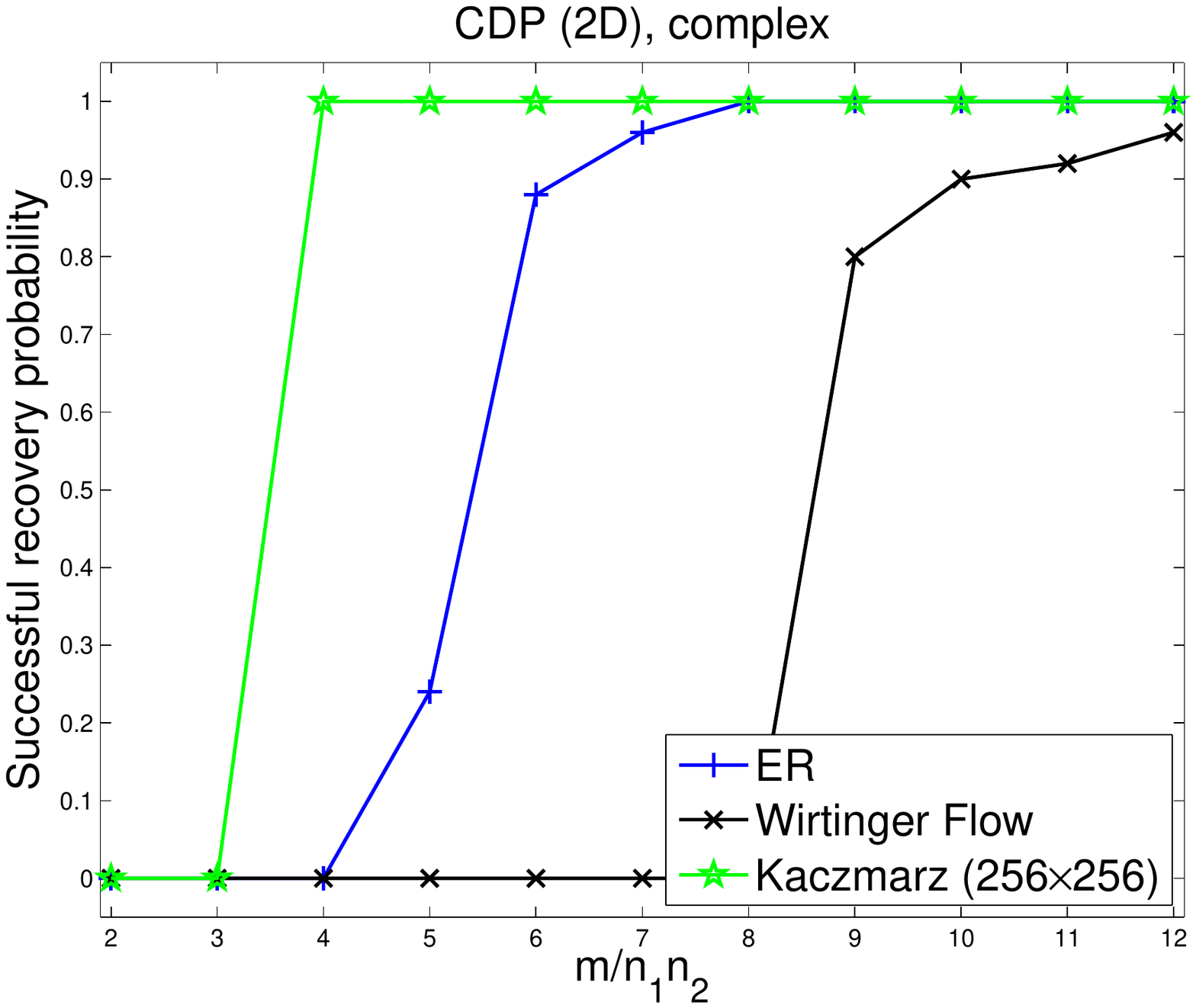}\\
(e) & (f)  
\end{tabular}
\caption{Empirical probability of successful recovery out of $50$ random trials; $1$D: $n=256$, $2D$: $n_1\times n_2=256\times 256$. In the coded 
diffraction model
$m=Ln$ or $m=Ln_1n_2$ with $L$ only taking integral values.\label{fig:prob_of_succ}}
\end{figure}
For the Gaussian real model (figure~\ref{fig:prob_of_succ} (a)), we observe that the recovery probability curves of the simple Kaczmarz method and the block Kaczmarz method with block sizes $32,64\mbox{ and }128$ are close to each other. They can successfully recover  $85\%$ of the test signals when $m=2.5n$. There are several
instances of successful recovery even when $m=2n$.
Note that there is a unique solution to \eqref{eq:problem_setup} in real case only if $m\geq 2n-1$ \cite{BaCaEd06a}.
%This observation is quite surprising by noticing that there is unique solution to Eq.~\eqref{eq:problem_setup} when $A$ and $x$ are real only if $m\geq 2n-1$ \cite{BaCaEd06a}. 
In contrast, ER and Wirtinger Flow respectively need $4.5n$ and $5n$ number of measurements for successful recovery with high probability. For the complex case (figure~\ref{fig:prob_of_succ} (b)),
the simple Kaczmarz method, the block Kaczmarz method with block sizes $32,64$ and ER have similar recovery probability curves which are slightly superior to Wirtinger Flow. They 
can successfully recover the test signals  with high probability when $m\approx 4n$ while the block Kaczmarz method with block size $128$ needs $m=5n$ number of measurements.

However,  the block Kaczmarz method with block size $n=256$ completely fails in both the real and complex cases. Here is a potential explanation for its failure.  
Example~\ref{example} indicates that  $x_{l+1}$ is generally not  the projection of $x_l$ onto the non-convex set \eqref{eq:proj_set} but only a heuristic approximation. Furthermore, theorem~\ref{prop2} shows that the deviation of $x_{l+1}$ from the optimal projection relies heavily on the condition numbers of the submatrices. In the block Kaczmarz method, $A_{\Gamma_r}$ is a $|\Gamma_r|\times n$ matrix.  Applying the Bai-Yin Law \cite{yinbai} heuristically shows that
\begin{equation}\label{eq:block_cond}
\kappa\lb A_{\Gamma_r}\rb\approx \frac{1+\sqrt{|\Gamma_r|/n}}{1-\sqrt{|\Gamma_r|/n}}.
\end{equation}
Therefore when $|\Gamma_r|=n/8,~4/n,\mbox{ and }n/2$, the corresponding conditions number of $A_{\Gamma_r}$ are approximately $2, ~3, \mbox{ and }6$. However, when $|\Gamma_r|=n$, 
the condition number of $A_{\Gamma_r}$ can be very large and hence $x_{l+1}$ cannot be a good approximation of the optimal projection of $x_l$ anymore even for very good approximations of the phase factor. 

Due to proposition~\ref{prop1}, only the block Kaczmarz method with each block corresponding to a unitary matrix $Q_\ell$ is tested for the unitary model. The Kaczmarz method can successful recover the real signals with high probability when $m\approx 3n$, while ER needs  $4.5n$ number of measurements (figure~\ref{fig:prob_of_succ} (c)). For the complex reconstruction problem (figure~\ref{fig:prob_of_succ} (d)), the recovery probability curve of the Kaczmarz method is higher than that of ER when $m\leq 4n$. Wirtinger Flow does not work for the unitary model in neither the real nor the complex cases.

For the $1$D coded diffraction model (figure~\ref{fig:prob_of_succ} (e)), the simplest Kaczmarz method has the largest probability of successful recovery when $L=3$. All the tested algorithms can recover the  signals
successfully with probability greater than $90\%$ when $L=4$, including the block Kaczmarz method with block size $n=256$. Notice that the recovery probability curves for the block 
Kaczmarz method with different block sizes are almost indistinguishable. 
In the coded diffraction model, the condition number of $A_{\Gamma_r}$ is less than $\sqrt{6}$ for any $|\Gamma_r|\leq n$. 
We also conduct tests for the $2$D coded diffraction model \footnote{For the other two models, $2D$ tests are essentially the same as $1$D tests after vectorizing the signals and measurement matrices.} for which the vectors $\hx\in\C^{n_1\times n_2}$ are sampled from Gaussian distribution
 $\N(0,I_{n_1\times n_2})+i\N(0,I_{n_1\times n_2})$. The set-up for the $2D$ coded diffraction model is similar to the $1$D case, but with the $1$D DFT replaced by the $2$D DFT.
For the Kaczmarz methods, only the block variant with each block corresponding to a coded diffraction is tested because for the coded diffraction model, the block Kaczmarz method can 
be potentially much faster than the simple Kaczmarz method as it can take advantage of the FFT. In the $2$D tests, we set $n_1=n_2=256$. Figure~\ref{fig:prob_of_succ} (f) shows that
the block Kaczmarz method is still able to successfully recover all the test signals when $L=4$, which is independent of the dimensionality of the signals. However, ER and Wirtinger Flow respectively need $6n$ and $10n$ number of measurements to achieve the high probability recovery, compared with the $4n$ number of measurements needed for the $1$D case. 
\subsubsection{Sensitivity to initial points}
For non-convex programming, the initialization procedure is very important and a good initial point can prevent the 
convergence to a local minimal solution. It has been reported in the literature that the performance of ER depends 
heavily on the initial points, see \cite{alt_min_pr, phasecut, wirtinger}. In this section, we will investigate  
the performance of the Kaczmarz methods when the initial points are generated 
randomly according to the standard Gaussian distribution. Figure~\ref{fig:prob_eig_rand} compares the recovery probability 
curves of  ER and the Kaczmarz methods under the random and spectral  initialization. In general,
 the Kaczmarz methods with the spectral initialization \eqref{eq:eig_init} have higher recovery curves than the random initialization when the number of measurements is small. However, it only requires at most $0.5n$ more number of measurements for the Kaczmarz methods with the random initialization to successfully recover all the test signals. 
In particular, the recovery curves of the block Kaczmarz methods for the  coded diffraction model are nearly indistinguishable 
for the two different initialization methods. In contrast, ER with the random initialization requires at least $2n$ more number of measurements 
to achieve the same recovery probability as the one with the spectral initialization. The randomly initialized ER fails all the tests for the real Gaussian
and unitary measurements. For the Gaussian measurement model, it is also worth noting that the block Kaczmarz method is less sensitive to the initial points than 
the simple Kaczmarz method.

Although the Kaczmarz methods are relatively less sensitive to the initial points than ER, starting from the initial points provided by
the spectral method can reduce the computation time of all the tested algorithms. Therefore without further mention, we will use the initial point $x_0$ in \eqref{eq:eig_init} for all the tested algorithms in the remainder of this section. 

\begin{figure}[htp]
\centering
\begin{tabular}{cc}
\vspace{0.1in}
\includegraphics[height=1.8in,width=2.2in,trim=0.7in 3in 0.7in 3in]{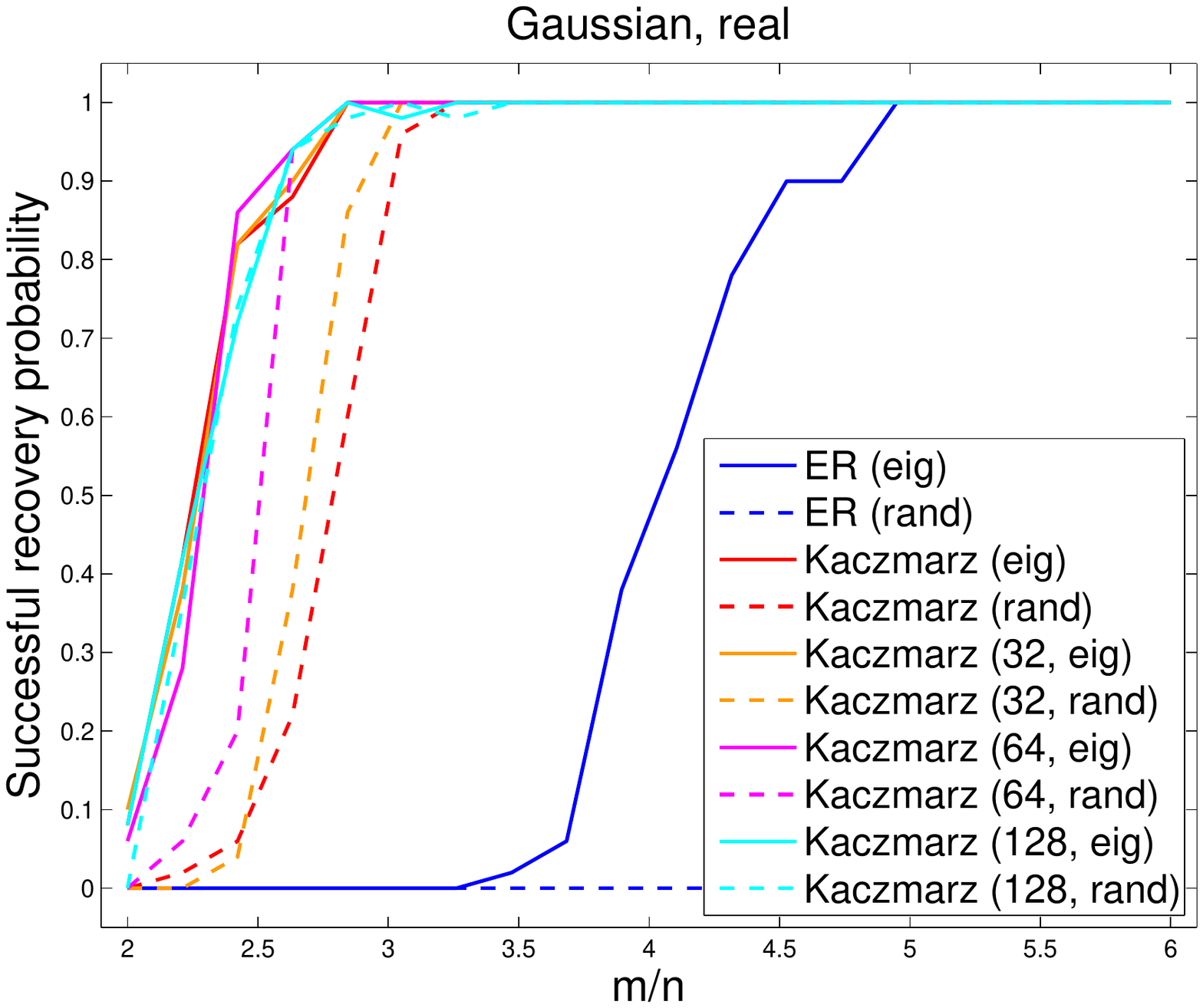}& 
\includegraphics[height=1.8in,width=2.2in,trim=0.7in 3in 0.7in 3in]{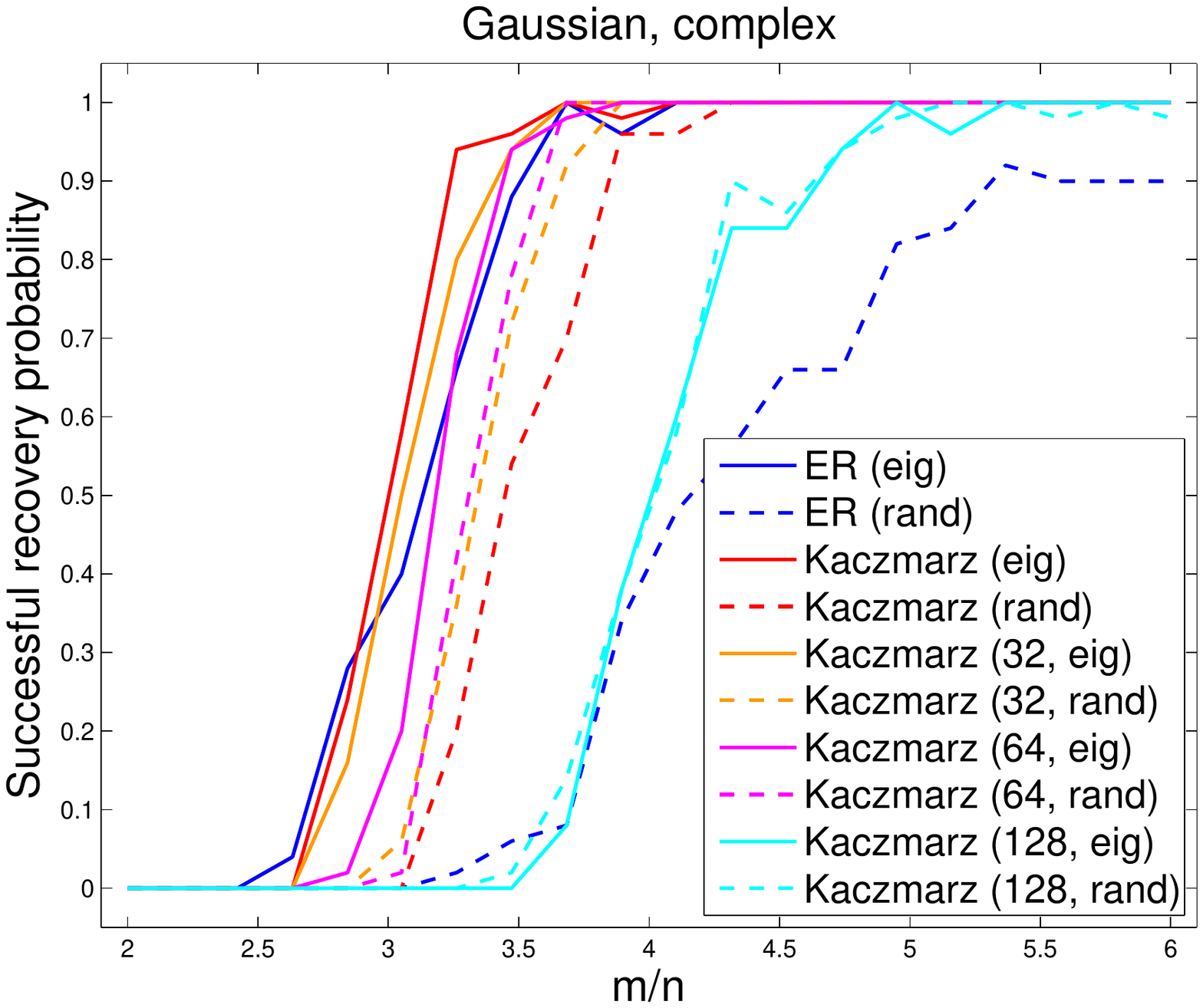}\\
\vspace{0.2in}
(a) & (b)\\
\vspace{0.1in}
\includegraphics[height=1.8in,width=2.2in,trim=0.7in 3in 0.7in 3in]{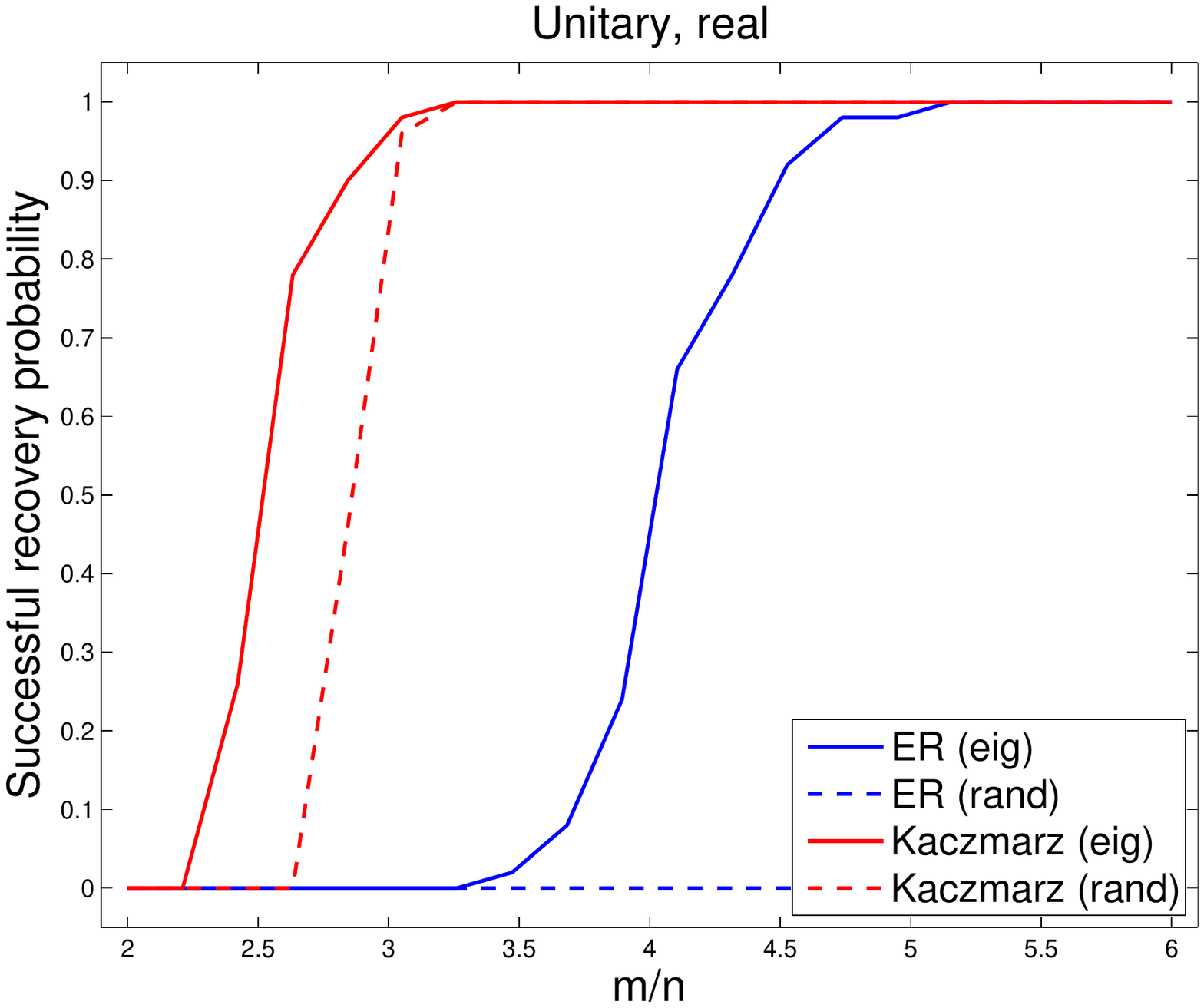}& 
\includegraphics[height=1.8in,width=2.2in,trim=0.7in 3in 0.7in 3in]{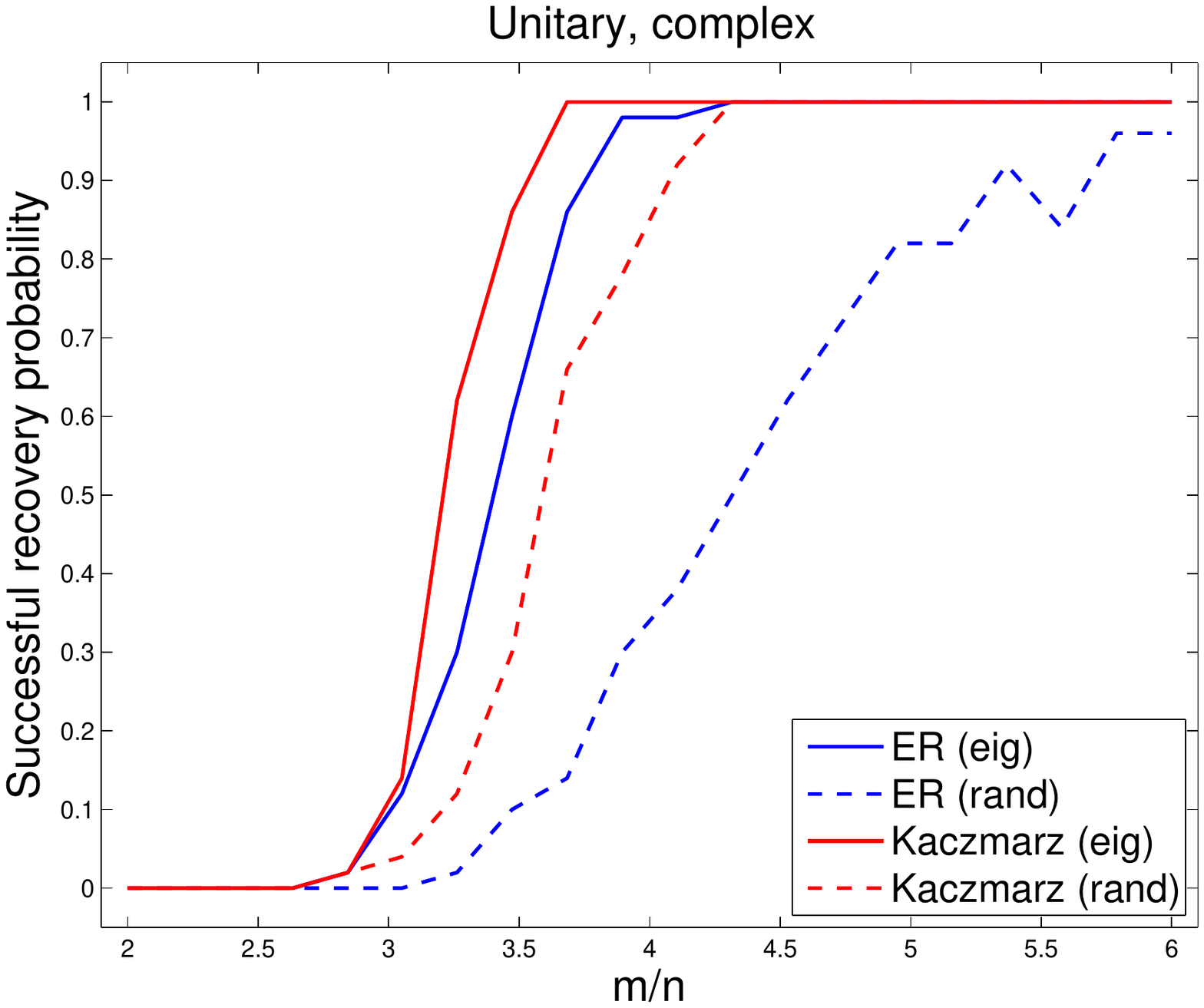}\\
\vspace{0.2in}
(c) & (d) \\
\vspace{0.1in}
\includegraphics[height=1.8in,width=2.2in,trim=0.7in 3in 0.7in 3in]{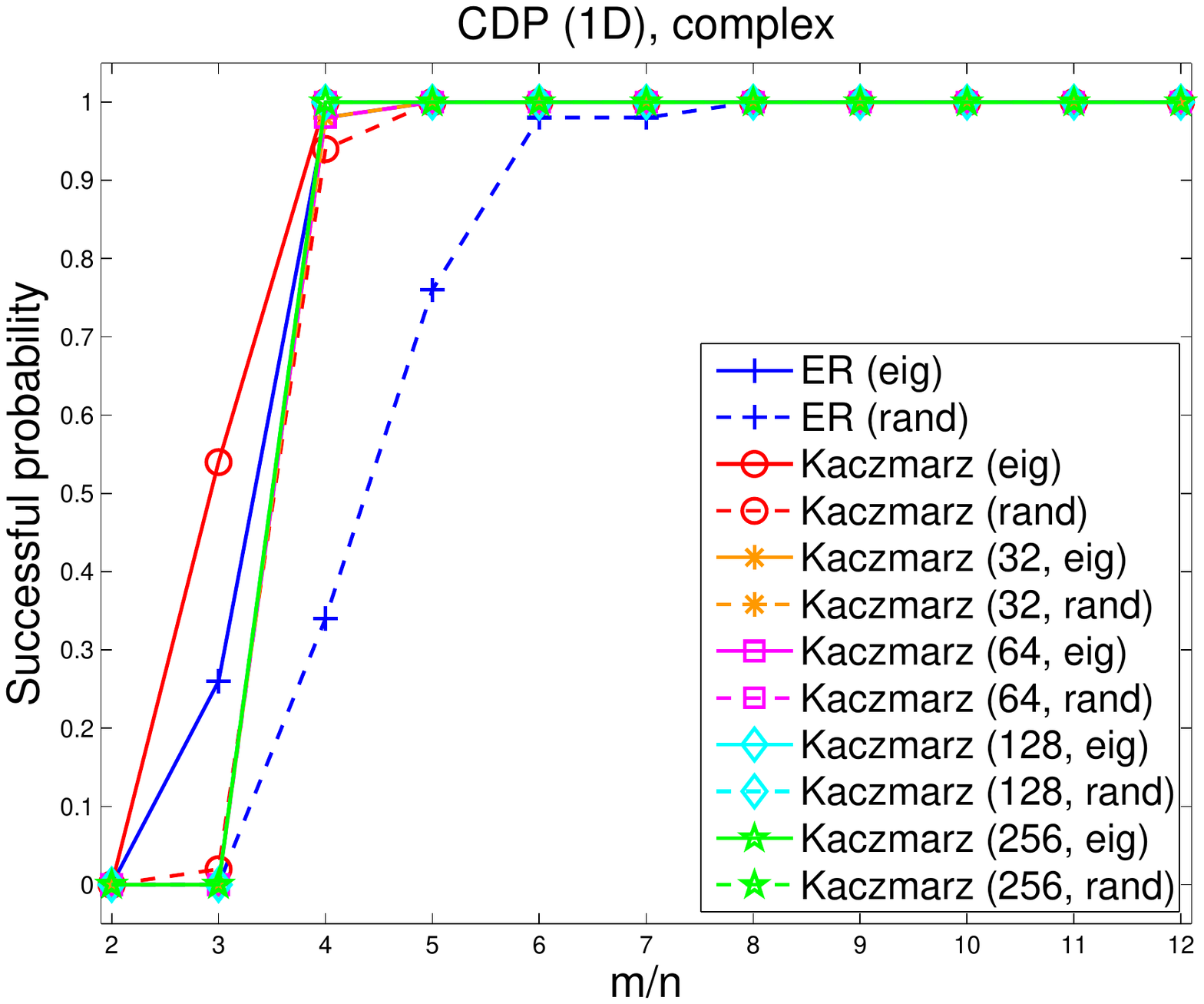}& 
\includegraphics[height=1.8in,width=2.2in,trim=0.7in 3in 0.7in 3in]{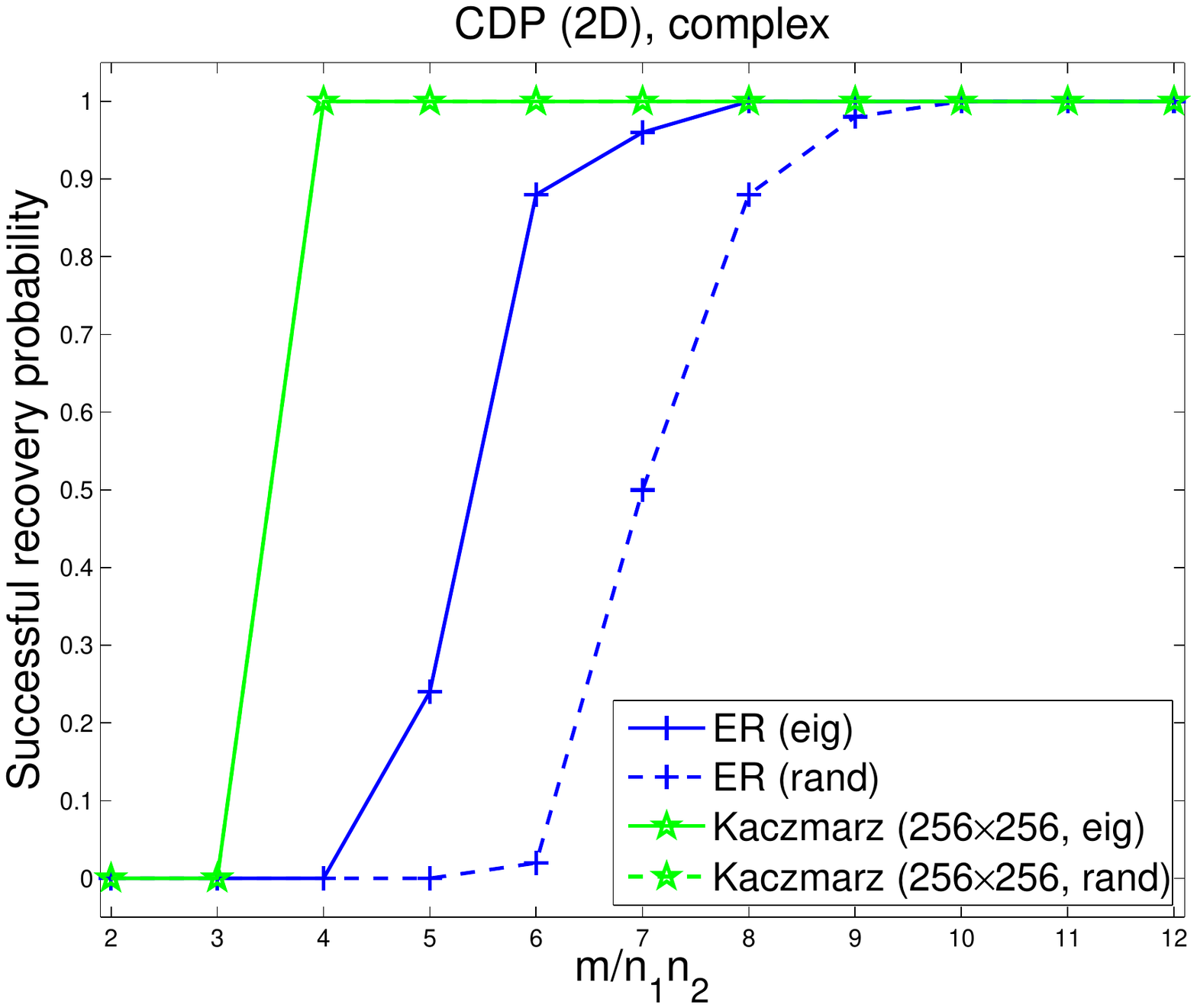}\\
(e) & (f)  
\end{tabular}
\caption{Empirical probability of successful recovery out of $50$ random trials for random and eigenvector  initializatons; $1$D: $n=256$, $2D$: $n_1\times n_2=256\times 256$. In the coded 
diffraction model
$m=Ln$ or $m=Ln_1n_2$ with $L$ only taking integral values.\label{fig:prob_eig_rand}}
\end{figure}

\subsubsection{Computation time} \label{sec:comp_time}
Next, we evaluate the computation time of the tested algorithms when solving the random problems to high accuracy. We conduct the
tests for the Gaussian and coded diffraction models with a sufficient number of measurements so that all the tested algorithms 
can succeed in recovering all of the {\em ten} randomly generated signals. The number of measurements $m$ is set  to  $6n$ for the 
Gaussian models and the coded diffraction $1$D model and $m$ is set to $12n_1n_2$ for the coded diffraction $2$D model. The maximum number of 
iterations (cycles) is set to $2500$ ($500$) as in section~\ref{sec:rate}. We set $\epsilon_1=10^{-10}$ for ER and Wirtinger Flow, $\epsilon_2=10^{-8}$ for the simple Kaczmarz method, and $\epsilon_2=10^{-9}$
for the block Kaczmarz method.
 The average number of iterations or cycles and the average computation time are listed in table~\ref{table}. First, it can be observed that it takes a similar number of cycles for the simple and block Kaczmarz methods to achieve  similar accuracy. However, the block Kaczmarz method is much faster in computation time as it can take advantage of the {\tt BLAS2} subroutines rather than {\tt BLAS1} for the
Gaussian models and the fast Fourier transform for the coded diffraction models. For the Gaussian complex model and the coded diffraction models, the block Kaczmarz method
are dramatically faster than ER and Wirtinger Flow because of its fast convergence rate. 

\begin{table}[htp]
\centering
\small
\caption{Computational results of ER, Wirtinger Flow and the Kaczmarz methods for solving the problems 
to high accuracy. In the block Kaczmarz method, the block size is selected to be $n/4$ for the Gaussian model, and respectively $n$ and $n_1\times n_2$ for the coded 
diffraction $1$D and $2$D models.}
\label{table}
\vspace{0.2cm}
\begin{tabular}{cccccccc}
\hline
& \multicolumn{3}{c}{ER} & &\multicolumn{3}{c}{Wirtinger Flow}\\
\cline{2-4}\cline{6-8}
%& \cline{2-4} & &\cline{6-8}
& \#its & time(s) & rel.err &&\#its & time(s)  & rel.err\\
\hline
Gaussian Real & 7& \textbf{0.032} & 8.68e-13 && 286&0.886&1.9e-10  \\
Gaussian Complex &144 &1.067 &2.0e-10 & & 911&7.011 &3.0e-10 \\
CDP 1D &157 &0.054 &1.99e-10 & &909 &0.301 & 2.94e-10\\
CDP 2D & 149&13.443 &1.3e-10 & &315 &22.796 & 1.91e-10\\
\hline
& \multicolumn{3}{c}{simple Kaczmarz} & &\multicolumn{3}{c}{block Kaczmarz}\\
\cline{2-4}\cline{6-8}
%& \cline{2-4} & &\cline{6-8}
& \#cycles & time(s)  & rel.err &&\#cycles & time(s) & rel.err\\
\hline
Gaussian Real &10 &0.469 &4.81e-13& & 8& 0.036& 6.81e-12 \\
Gaussian Complex & 24&1.643 &1.14e-10 & & 26&\textbf{0.274}&3.04e-10  \\
CDP 1D & 24& 1.927&7.3e-11 & &25 & \textbf{0.009}&3.13e-10 \\
CDP 2D &- &- & -& &15 &\textbf{1.478} &2.41e-11 \\
\hline
\end{tabular}
\end{table}

\subsubsection{Robustness to additive noise} 
With the same number of measurements as in  section~\ref{sec:comp_time}, we further explore the performance of the algorithms under noisy measurements of the form
\begin{equation}
y = \max(|A\hx|^2+\varepsilon\ln |A\hx|^2\rn_2\cdot e,0),
\end{equation}
where    $\varepsilon>0$ denotes the noise level and $e\in\R^m$ is uniformly distributed on the unit sphere. We test the algorithms with nine different noise 
levels from $10^{-5}$ to $0.1$. The  plots of the relative errors against the noise levels are shown in figure~\ref{fig:robust_to_noise} for the Gaussian 
and coded diffraction models. The plots show clearly the desirable linear scaling between the noise levels and the relative errors for all the tested algorithms and tested models.
\begin{figure}[htp]
\centering
\begin{tabular}{cc}
\vspace{0.1in}
\includegraphics[height=1.8in,width=2.2in,trim=0.7in 3in 0.7in 3in]{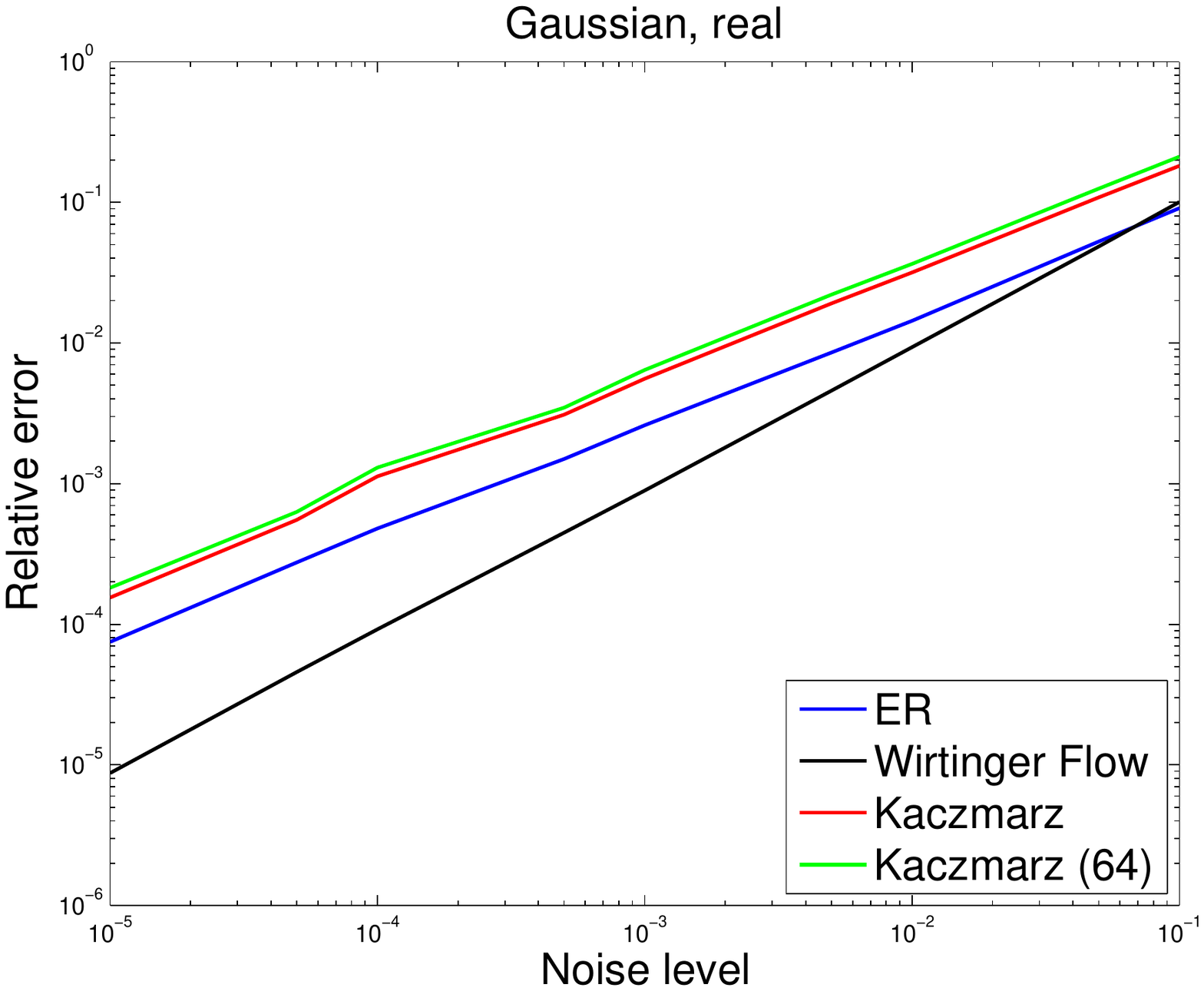}& 
\includegraphics[height=1.8in,width=2.2in,trim=0.7in 3in 0.7in 3in]{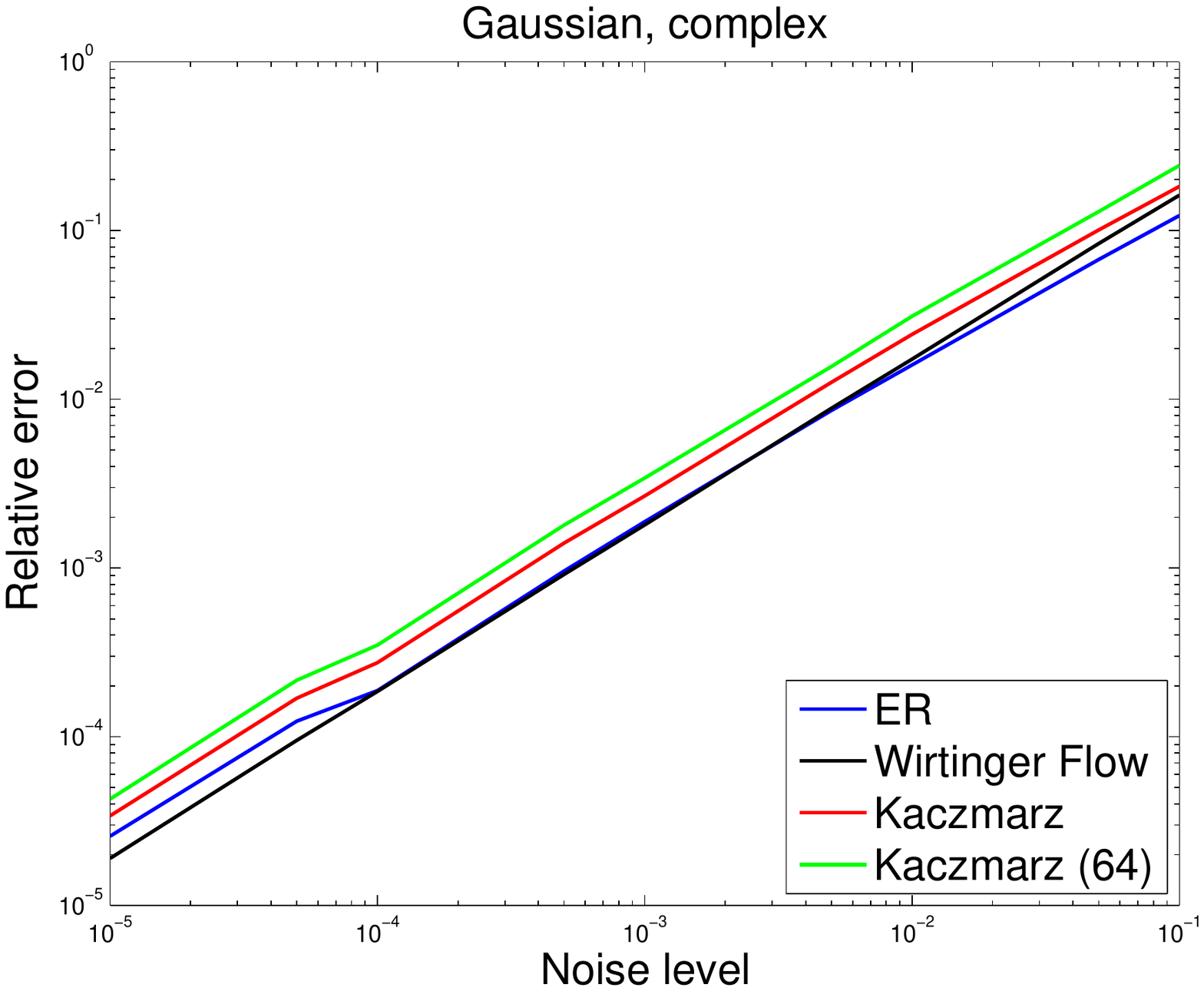}\\
\vspace{0.2in}
(a) & (b)\\
\vspace{0.1in}
\includegraphics[height=1.8in,width=2.2in,trim=0.7in 3in 0.7in 3in]{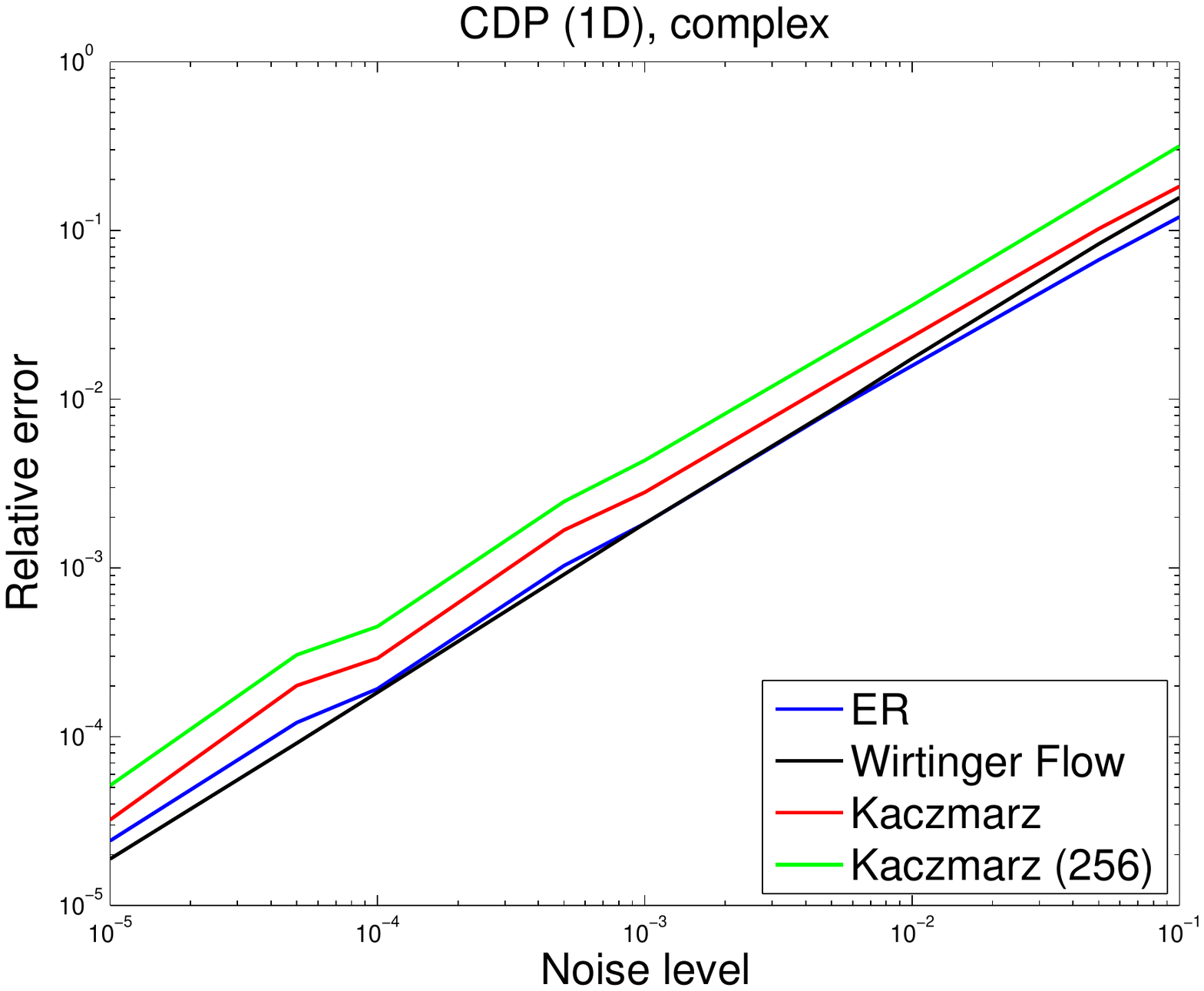}& 
\includegraphics[height=1.8in,width=2.2in,trim=0.7in 3in 0.7in 3in]{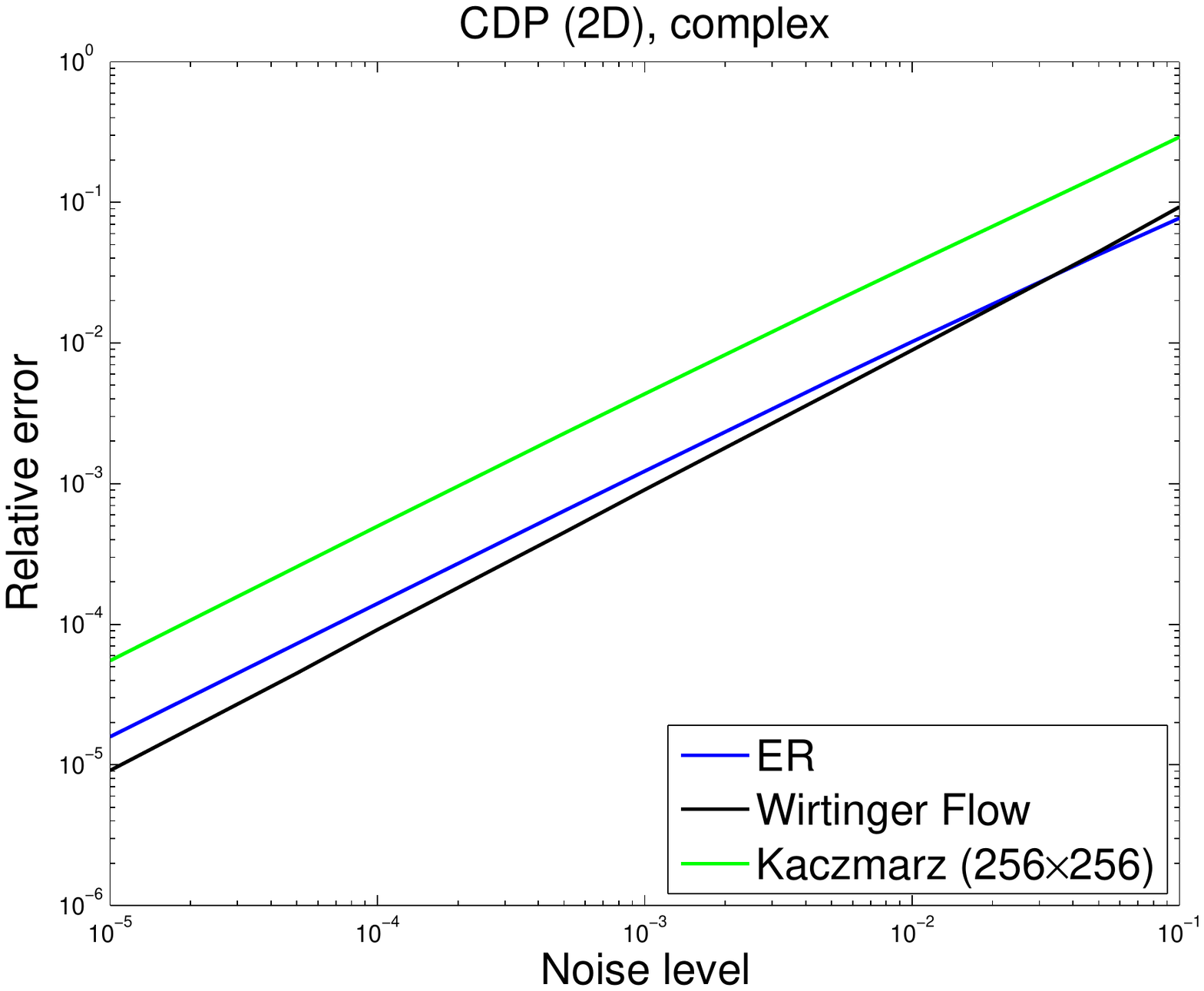}\\
\vspace{0.1in}
(c) & (d)
\end{tabular}
\caption{Log-log plots of relative errors vs noise levels; $1$D $n=256$, $2$D: $n_1=n_2=256$.\label{fig:robust_to_noise}}
\end{figure}
 
\subsection{Performance on molecules and natural images}
\begin{figure}[htp]
\centering
\begin{tabular}{cc}
%\vspace{0.1in}
\includegraphics[height=1.5in,width=2.2in,trim=0.7in 2.5in 0.7in 3.45in]{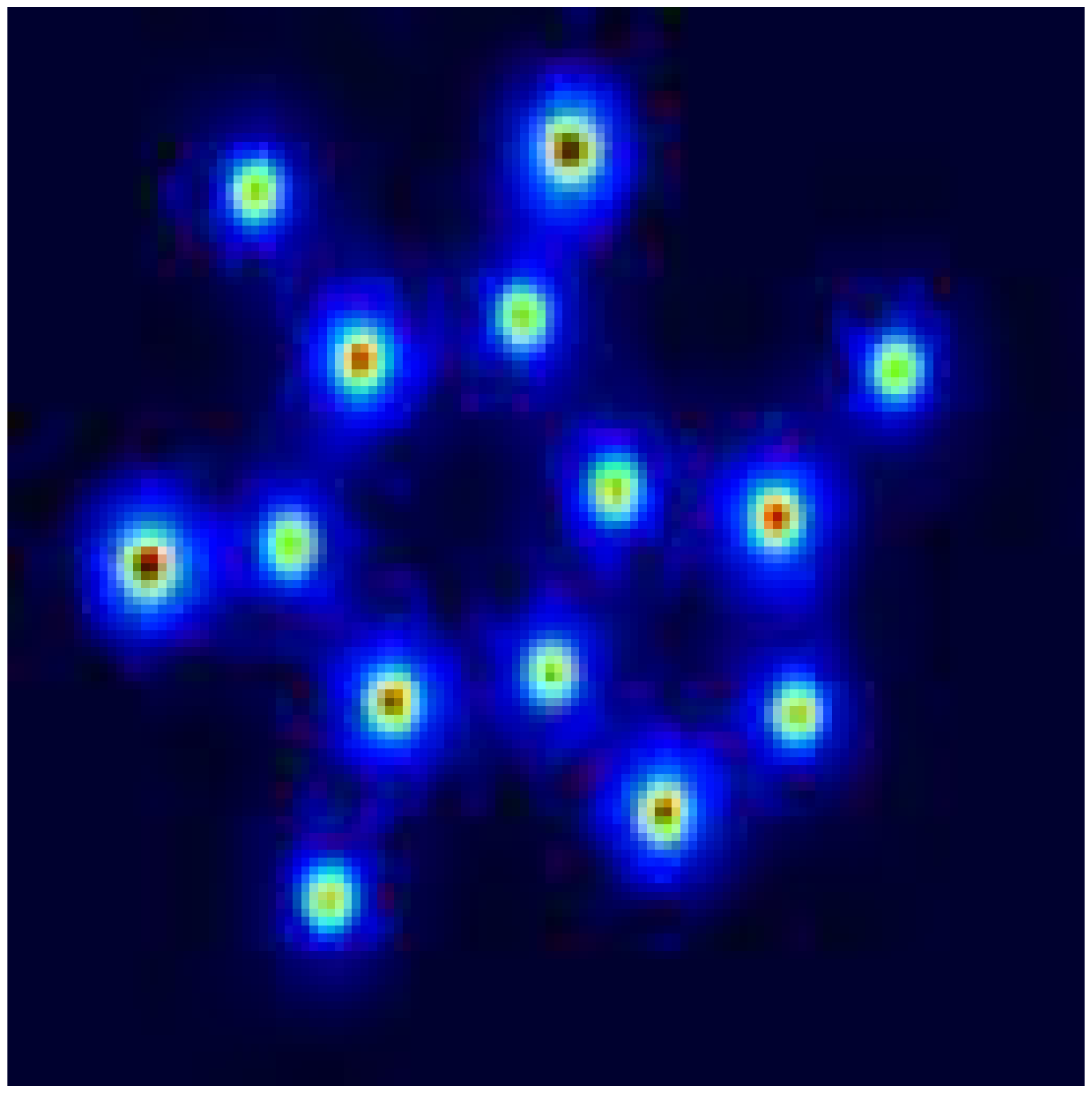}& 
\includegraphics[height=1.8in,width=2.5in,trim=0.7in 3in 0.7in 3in]{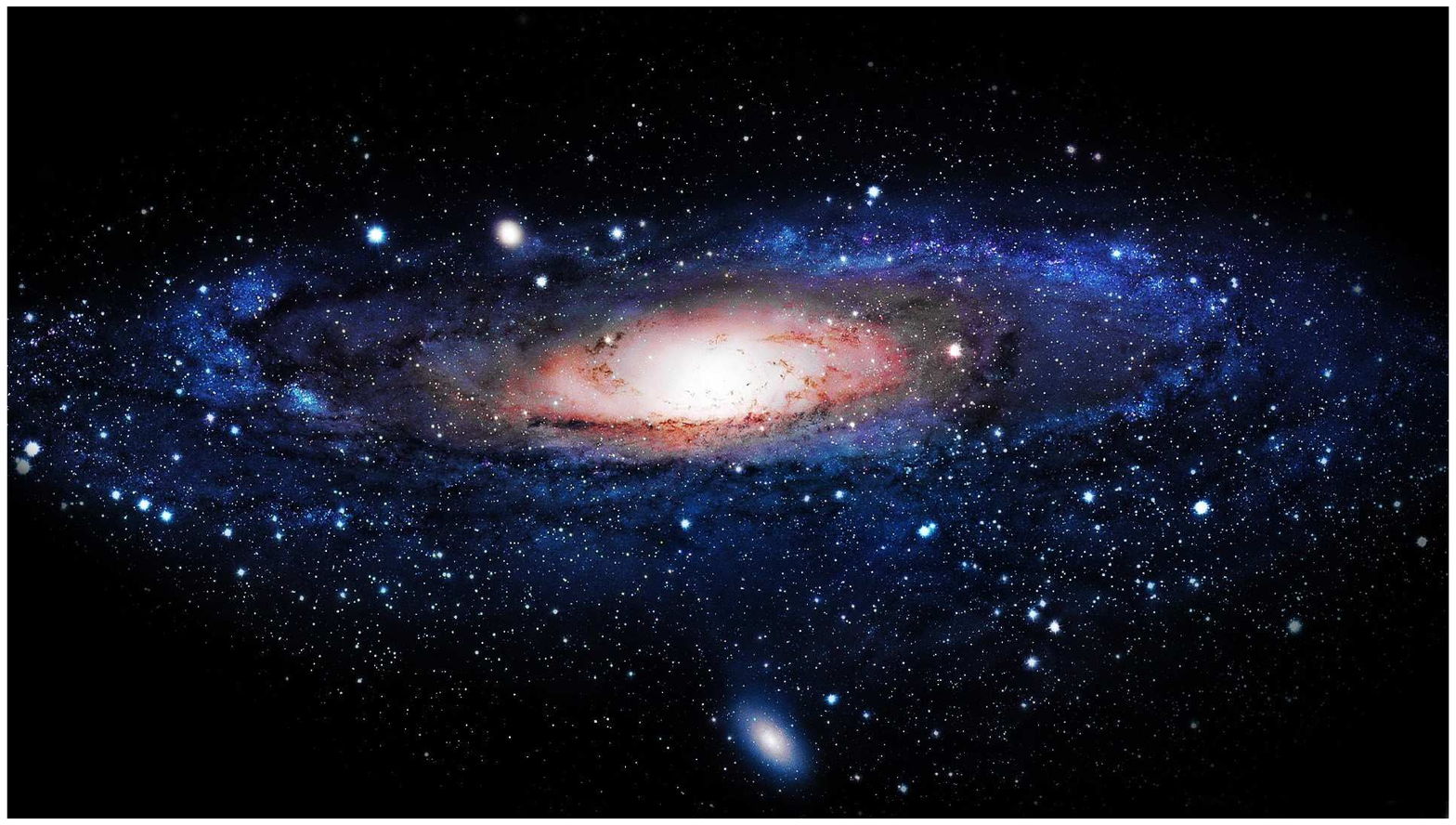}\\
(a) & (b)
\end{tabular}
\caption{(a) projection of the $3$D Caffeine molecule's density map along $z$-axis, size: $128\times 128$; (b) photograph of the Milky Way Galaxy, size $1080\times 1920$.\label{fig:photos}}
\end{figure}

In this subsection, we evaluate the performance of the Kaczmarz methods on real images and molecules for the coded diffraction model. Based on the empirical observations in section~\ref{sec:rand_tests}, apparently the block Kaczmarz method
with each block corresponding to a coded diffraction pattern is highly recommended for this task compared to the other variants. So we only test this variant compared with ER and Wirtinger Flow, using the same stopping criteria as in section~\ref{sec:comp_time}. The algorithms are tested for four different numbers of coded diffraction patterns $L=4,8,12,16$. For a fixed $L$, we repeat the tests {\it ten} times.

We first run the algorithms on the projection of the $3$D Caffeine molecule's density map onto the $xy$-plane along the $z$-axis. Figure~\ref{fig:photos} (a) makes a plot of this projection, which is a $128\times 128$ matrix.  For the details on the computation of the projection, we refer the reader to \cite{wirtinger, phasecut_image}. 
Next we  test the algorithms on the Milky Way Galaxy photograph of size $1080\times 1920$, see figure~\ref{fig:photos}~(b). Since it is an RGB photograph, we run the algorithms on each R, G, B channels independently. The computational results for both the projection of the molecule's density map and the Milky Way Galaxy are listed in table~\ref{table2}. For the Galaxy, we consider a random test to be successful if the algorithm can successfully reconstruct all the three R, G, B images for the generated coded diffraction patterns. The number of iterations (cycles) and the computation time for the Galaxy are average results over all the successful recoveries and the three R, G, B channels.

\begin{table}[htp]
\centering
\small
\caption{Computational results of ER, Wirtinger Flow and the block Kaczmarz method
for reconstructing the $3$D Caffeine molecule's density projection and the RGB image Milky Way Galaxy to high accuracy. Four different numbers of coded diffraction patterns $L=4,8,12,16$ are tested.\label{table2}}
\vspace{0.2cm}
\begin{tabular}{ccccccccccc}
\hline
\multicolumn{2}{c}{} & \multicolumn{4}{c}{Caffeine}& & \multicolumn{4}{c}{Galaxy}\\
\cline{3-6}\cline{8-11}
\multicolumn{2}{c}{} & $L=4$ & $L=8$ & $L=12$ & $L=16$& & $L=4$ & $L=8$ & $L=12$ & $L=16$\\
\hline
\multirow{4}{*}{ER} & \#succ &0 &6 &10 &10 & &0 &2 &2 &10\\
 & rel.err &- &1.6e-10 &1.4e-10 &1.3e-10 & &- &1.7e-10 &1.3e-10 &1.2e-10\\
 & \#its &- &227 &95 &69 & &- &904 &353 &154\\
 & time(s) &- &3.26 &2.11 &2.0 & &- &2038 &858.3 &686.2\\
\hline
 & \#succ &0 &3 &10 &10 & &0 &0 &4 &10\\
Wirtinger & rel.err &- &1.1e-10 &2.0e-10 &1.7e-10 & &- &- & 1.9e-10&1.7e-10\\
Flow & \#its &- &1057 &279 &233 & &- &- &467 &277\\
 & time(s) &- &12.4 &4.8 &5.4 & &- &- &1362.8 &1063.2\\

\hline
 & \#succ &10 &10 &10 &10 & &10 &10 &10 &10\\
block & rel.err &7.8e-10 &1.1e-10 &2.0e-11 &9.4e-12 & &7.2e-10 &1.3e-10 &3.4e-11 &1.2e-11\\
Kaczmarz & \#cycles &169 &23 &12 &8 & &296 &28 &19 &13\\
 & time(s) &\textbf{1.4} &\textbf{0.4} &\textbf{0.33} &\textbf{0.3} & &\textbf{356} &\textbf{69.7} &\textbf{69.8} &\textbf{65.3}\\

\hline
\end{tabular}
\end{table}

First the table shows that $L=4$ is sufficient for the block Kaczmarz method to successfully reconstruct both the molecule's density map projection and the Galaxy from the magnitude measurements, which coincides with our observations in the random simulations, see figure~\ref{fig:prob_of_succ} (f). For both ER and Wirtinger Flow, reconstructing the  sophisticated Galaxy requires more numbers of the coded diffraction patterns than reconstructing the molecule's density map projection. Regarding to the computation time, the block Kaczmarz method is overall ten times faster than ER and twenty times faster than Wirtinger Flow.
%-------------------------------------
\section{Proof of theorem~\ref{thm:simple}}\label{sec:proof}
The proof of theorem~\ref{thm:simple} uses a result from \cite{Need10a}, which is stated below.
\begin{lemma}\label{lem1}
Let $Ax=y$ be a linear system with $A\in\C^{m\times n}$ being standardized. Let $x_{l-1}$ be any vector in $\C^n$ and $x_l$ be the 
random projection of $x_{l-1}$ computed from the randomized Kaczmarz method for the inconsistent linear 
system $Ax = y+e$. Let $\hx$ be the solution to  $Ax=y$. Then we have
\begin{equation*}
\E\lsb \ln x_l-\hx\rn_2^2\rsb\leq \lb 1-\frac{\sigma^2_{\min}(A)}{m}\rb\ln x_{l-1}-\hx\rn_2^2+R^2,
\end{equation*}
where $R=\max_{1\leq r\leq m}|e_r|$.
\end{lemma}
%--------------------------
\begin{proof}[Proof of theorem~\ref{thm:simple}]
Denote by $x_{l}^r$ the  estimate obtained when the $r$th row is selected, that is 
\begin{equation*}
x_{l}^r = x_{l-1}+\frac{\sqrt{y_r}e^{i\theta_{l-1}^r}-\la a_r,x_{l-1}\ra}{\ln a_r\rn_2^2}a_r, 
\end{equation*}
where $\theta_{l-1}^r = \angle\la a_r,x_{l-1}\ra,~r=1,\cdots,m$. Then at the $(l-1)$th iteration one has $\P\lcb x_{l}=x_{l}^r\rcb=\frac{1}{m}$. 
Expressed differently, $x_l$ can be viewed as the solution obtained by applying one iteration of the randomized Kaczmarz method  for the inconsistent linear system
\begin{equation}\label{eq:incons}
Ax = \sqrt{y}\odot e^{i\theta_{l-1}}~\mbox{ where }e^{i\theta_{l-1}} = \begin{bmatrix}e^{-i\theta_{l-1}^1}&\cdots&e^{-i\theta_{l-1}^m}\end{bmatrix}^*.
\end{equation}
In order to apply lemma~\ref{lem1}, we will form a consistent linear system in each iteration as follows.
Define $\phi_{l-1}=\argmin_{\theta\in[0,2\pi)}\ln x_{l-1}-\hx e^{i\theta}\rn_2$ and $\psi_{l-1}^r=\angle\la a_r,\hx e^{i\phi_{l-1}}\ra$. Then 
the linear system 
\begin{equation}\label{eq:cons}
Ax = \sqrt{y}\odot e^{i\psi_{l-1}}~\mbox{ where }e^{i\psi_{l-1}} = \begin{bmatrix}e^{-i\psi_{l-1}^1}&\cdots&e^{-i\psi_{l-1}^m}\end{bmatrix}^*
\end{equation}
is consistent and the solution to \eqref{eq:cons} is given by $\hx e^{i\phi_{l-1}}$. Rewrite \eqref{eq:incons} as 
\begin{equation*}
Ax = \sqrt{y}\odot e^{i\psi_{l-1}} + \underbrace{\sqrt{y}\odot\lb e^{i\theta_{l-1}}-e^{i\psi_{l-1}}\rb}_{e}.
\end{equation*}
Then applying lemma~\ref{lem1} gives 
\begin{equation*}
\E\lsb\ln x_l-\hx e^{i\phi_{l-1}}\rn_2^2\rsb\leq \lb 1-\frac{\sigma^2_{\min}(A)}{m}\rb\ln x_{l-1}-\hx e^{i\phi_{l-1}}\rn_2^2 +R^2,
\end{equation*}
where 
\begin{equation*}
\max_{1\leq r\leq m} \left|\sqrt{y_r}\lb e^{i\theta_{l-1}^r}-e^{i\psi_{l-1}^r}\rb\right|\leq
2\max_{1\leq r\leq m}\sqrt{y_r}:=R,
\end{equation*}
and the expectation is conditioned on $x_{l-1}$.
Since for every $1\leq r\leq m$, $\dist(x_l^r,\hx)\leq\ln x_l^r-\hx e^{i\phi_{l-1}}\rn_2$, we have
\begin{eqnarray*}
\E\lsb\dist^2(x_l,\hx)\rsb&\leq&\E\lsb\ln x_l-\hx e^{i\phi_{l-1}}\rn_2^2\rsb\\
&\leq&\lb 1-\frac{\sigma^2_{\min}(A)}{m}\rb\ln x_{l-1}-\hx e^{i\phi_{l-1}}\rn_2^2 +R^2\\
&=&\lb 1-\frac{\sigma^2_{\min}(A)}{m}\rb\dist^2(x_{l-1},\hx) +R^2,
\end{eqnarray*}
where the last equality follows from the definition of $\dist(\cdot,\cdot)$. Taking full expectation on both sides and applying 
the resulted relationship repeatedly yields
\begin{eqnarray*}
\E\lsb\dist^2(x_l,\hx)\rsb&\leq& \lb 1-\frac{\sigma^2_{\min}(A)}{m}\rb^l\dist^2(x_0,\hx)+\sum_{k=0}^{l-1}\lb 1-\frac{\sigma^2_{\min}(A)}{m}\rb^kR^2\\
&\leq&\lb 1-\frac{\sigma^2_{\min}(A)}{m}\rb^l\dist^2(x_0,\hx)+\frac{mR^2}{\sigma^2_{\min}(A)},
\end{eqnarray*} 
which concludes the proof by reintroducing the value of $R$.
\end{proof}
%-------------------------------------
\section{Conclusion and future directions}\label{sec:conclusion}
Phase retrieval is an important topic which has received intensive investigations recently.
This manuscript develops the Kaczmarz methods for the generalized phase retrieval problem.
%The basic properties are discussed for the block and randomized Kaczmarz methods. 
The empirical results demonstrate that the Kaczmarz methods are superior to ER and Wirtinger Flow
 in terms of the successful recovery probabilities and overall computation time. For the block Kaczmarz 
method, we show that the condition numbers of the submatrices play a key role in 
successful recovery in theorem~\ref{prop2}. 
 To the best of our knowledge, this is the first paper which %extends the Kaczmarz methods
%from solving linear systems to solve the phase retrieval problem. 
suggests applying the Kaczmarz methods to solve the systems of phaseless equations.

A central question that remains to be answered is how many measurements  are needed for the Kaczmarz methods to successfully find
 the solution to the generalized phase retrieval problem. For the real case when the signals and measurement vectors are all 
real-valued, the Kaczmarz methods for phase retrieval can reduce to the Kaczmarz methods for linear equations if the 
initial point is close enough to the true solution $\hx$ because of the separation of $\hx$ and $-\hx$. So in principle, $O(n\log n)$ number of 
measurements are sufficient for the Gaussian real model, see the remark after \eqref{eq:eig_init}. However, this is not true 
for the complex measurements because $\hx e^{i\theta}$ is continuous with respect to $\theta\in[0,2\pi)$.  For the randomized Kaczmarz methods,
notice that the proof of theorem~\ref{thm:simple} does not use any information provided by the phase selection heuristic.  A recovery guarantee is 
likely to require an analysis of how well the phase selection heuristic approximates the phase of the true solution in each iteration.

This manuscript opens a door of applying the existing techniques in the accelerated Kaczmarz methods for linear equations to the phase retrieval problem, for example by introducing  
relaxation. With the relaxation, the update rule in the simple Kaczmarz method (Alg.~\ref{alg:simp_kacz}) becomes
\begin{equation*}
x_{l+1} = x_l+\lambda_l\frac{\sqrt{y_r}e^{i\theta_l}-\la a_r,x_l\ra}{\ln a_r\rn_2^2}a_r.
\end{equation*}
It is suggested in \cite{StVe09b} that the randomized simple Kaczmarz method for linear systems can be accelerated for Gaussian measurement matrices if relaxation parameter is set
to $\lambda_l=1+n/m$ in each iteration. The numerical simulations show that the same selection of the relaxation parameter can also accelerate the simple Kaczmarz method for phase 
retrieval, see figure~\ref{fig:kacz_vs_relaxed}. The Kaczmarz methods for phase retrieval are designed by extending one of the existing linear solvers to solve the system of 
quadratic equations. Other typical linear solvers such as conjugate gradient method may well be similarly effective.

\begin{figure}[t]
\centering
\includegraphics[height=1.5in,width=3.2in,trim=0.7in 3in 0.7in 3in]{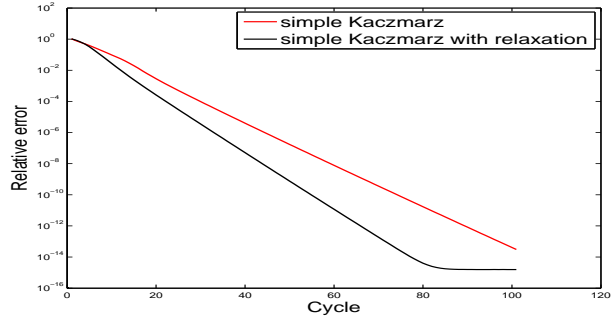}
\caption{Convergence rate comparison of the simple Kaczmarz method  with and without relaxation for the Gaussian complex model
with $n=256$ and $m=4n$. The relaxation parameter $\lambda_l$ is selected to be $1+n/m$ in each iteration. Averaged over $50$ random tests.\label{fig:kacz_vs_relaxed}}
\end{figure}
As demonstrated in the literature \cite{fienup1,fienup2}, ER often works much better if  a priori knowledge about the signals is incorporated, such as real-valued, nonnegative and sparsity.  
Despite the already very good performance of the Kaczmarz methods for our test problems, it is worth investigating whether their performance can be further improved by exploring the structures of the signals.
Finally, it may also be possible to extend the Kaczmarz methods to other related problems, such as blind deconvolution \cite{blind_conv} and self-calibration \cite{self_cal}.
\section*{Acknowledgments}
The author would like to thank Thomas Strohmer for helpful suggestion. He would also like to thank the reviewers for useful comments which have improved the manuscript.  
%-------------------------------------
\bibliographystyle{unsrt}
\bibliography{phase}
\end{document}